\documentclass[reqno,11pt]{amsart}
\usepackage{amsfonts,amsmath,amssymb}
\usepackage{graphicx}

\usepackage{color}
\usepackage{mathtools}
\DeclarePairedDelimiter{\ceil}{\lceil}{\rceil}
\DeclarePairedDelimiter{\floor}{\lfloor}{\rfloor}
\usepackage[margin=1.35in]{geometry}
\numberwithin{equation}{section}

\def\be{{\beta}}
\def\ga{\gamma}
\def\Ga{\Gamma}
\def\de{\delta}

\def\ep{\epsilon}

\def\la{\lambda}

\def\ze{\zeta}

\def\al{\alpha}

\def\hat{\widehat}

\def\bar{\overline}

\def\Li{\operatorname{Li}}

\def\L{\mathcal L}

\newtheorem{theorem}{Theorem}[section]

\newtheorem{remark}[theorem]{Remark}

\newtheorem{lemma}[theorem]{Lemma}

\newtheorem{prop}[theorem]{Proposition}
\newtheorem{conjecture}[theorem]{Conjecture}
\newtheorem{defi}[theorem]{Definition}

\newcommand{\RR}{\mathbb{R}}
\newcommand{\C}{\mathbb{C}}
\newcommand{\Z}{\mathbb{Z}}

\newcommand{\T}{\mathbb{T}}

\newcommand{\dy}{\partial_y}

\newcommand{\sign}{{\rm sgn}\thinspace}

\newcommand{\te}{\theta}

\newcommand{\vp}{\varphi}

\newcommand\TT{\mathbb T}
\allowdisplaybreaks[4]

\begin{document}

\title{Convexity of Whitham's highest cusped wave}

\author{Alberto Enciso}
\address{Instituto de Ciencias Matem\'{a}ticas}
\email{aenciso@icmat.es}

\author{Javier G\'omez-Serrano} 
\address{Princeton University}
\email{jg27@math.princeton.edu}

\author{Bruno Vergara}
\address{Instituto de Ciencias Matem\'{a}ticas}
\email{bruno.vergara@icmat.es}

\thanks{}

\begin{abstract}
We prove the existence of a periodic traveling wave of extreme form of
the Whitham equation that has a convex profile between consecutive
stagnation points, at which it is known to feature a cusp of exactly
$C^{1/2}$~regularity. The convexity of Whitham's highest cusped
wave had been conjectured by Ehrnstr\"om and Wahl\'en.
\end{abstract}

\maketitle

\section{Introduction}

Whitham's equation~\cite{Whitham} is a nonlinear, nonlocal shallow water wave model in one space
dimension that reads as
\begin{equation}\label{Whithamt}
\partial_t v + \partial_x(L v+v^2)=0\,,
\end{equation}
where~$L$ is the Fourier multiplier defined in terms of the full
dispersion relation for gravity water waves
$m(\xi):=(\tanh\xi/\xi)^{1/2}$,
\[
\widehat{Lf}(\xi):= m(\xi)\, \widehat f(\xi)\,.
\]
Whitham proposed this equation in 1967 as an alternative to the
well-known KdV equation, as the latter does not accurately describe
the dynamics of short waves. 

The key feature of Whitham's equation is its very weak dispersion,
which is due to the fact that the symbol~$m(\xi)$ has a completely
different behavior for large frequencies than equations such as KdV,
whose corresponding Fourier multiplier is precisely defined by the
second-order Taylor series of~$m(\xi)$. This very weak dispersion
allows Whitham's equation to exhibit both smooth periodic and solitary
solutions on the one hand~\cite{EGW,EK1,EK2}, and singular solutions on the
other. 

Singular solutions for the Whitham equation appear as wave
breaking~\cite{Constantin,Hur} (i.e., as bounded solutions
to~\eqref{Whithamt} whose derivative blows up in finite time) and as
traveling waves with sharp crests, which are only of $C^{1/2}$~regularity. In this paper we will be
concerned with the latter, whose existence was conjectured
some forty years ago by Whitham~\cite{Whitham} and established by
Ehrnstr\"om and Wahl\'en~\cite{EW} just recently. Interestingly, if one replaces the
Whitham equation by a related fully dispersive model that contains
both branches of the full Euler dispersion relation instead of just
one, non-smooth traveling waves have been found too~\cite{EJC}, but
the solutions are in $C^\al$ for all~$\al<1$.

Let us elaborate on the existence of sharp crests. With the ansatz $v(x,t):= \vp(x-\mu
t)$, the study of traveling waves for the Whitham equation reduces to
the analysis of the equation
\begin{equation}
  \label{Wphi}
  L\vp-\mu\vp+\vp^2=0\,,
\end{equation}
where the positive constant~$\mu$ represents the speed of the
traveling wave. Whitham himself conjectured~\cite[p.~479]{Whitham2}
that the equation should admit traveling
waves with a sharp crest, and provided a heuristic argument suggesting
that the crest should be cusped with $\vp(x) \sim
\frac\mu2-c|x|^{1/2}$. 

Ehrnstr\"om and Wahl\'en's proof of this conjecture~\cite{EW} is
based on a remarkable global bifurcation argument, where cusped
solutions of any period were shown to exist by continuing off a
local branch of small amplitude periodic traveling waves bifurcating
from the zero state. These solutions were shown to be smooth away from their
highest point (the crest) and behave like $|x|^{1/2}$ near the crest,
so their sharp H\"older regularity is~$C^{1/2}$. These authors also
conjectured that, just as in the celebrated case of the highest traveling water
waves (which present a corner of 120 degrees)~\cite{Acta,PT}, Whitham's highest cusped waves
must be convex between the crests:

\begin{conjecture}{(Convexity of Whitham's highest cusped
    wave~\cite[p.~4]{EW})}\label{conj}
    Whitham's highest wave $\vp$ is everywhere convex and its
    asymptotic behavior at~$0$ is
\begin{equation*}
  \vp(x)=\frac\mu 2 -\sqrt{\frac{\pi }{8}}|x|^{1/2}+o(|x|)\,.
\end{equation*}
\end{conjecture}

Our objective in this paper is to prove this conjecture. For concreteness,
we take Whitham's highest cusped wave of period $2\pi$, so $\vp$ can
be assumed to be a function on $\TT:=
\RR/2\pi\mathbb Z$. Our main result can be stated as follows:

\begin{theorem}\label{mainthm}
 The $2\pi$-periodic highest cusped traveling wave $\vp\in
 C^{1/2}(\TT)$ of the Whitham equation
is a convex function and behaves asymptotically as
\begin{equation}
  \label{asymp}
  \vp(x)=\frac\mu 2 -\sqrt{\frac{\pi }{8}} |x|^{1/2}+O(|x|^{1+\eta})
\end{equation}
for some $\eta>0$.
Furthermore, $\vp$ is even and strictly decreasing on the interval $[0,\pi]$.
\end{theorem}

The proof of Theorem~\ref{mainthm} is rather involved and relies in
part on computer-assisted estimates. We start off by
noticing that the function
\begin{equation}
  \label{udef}
  u(x)=\frac \mu 2-\vp(x)\,
\end{equation}
satisfies an equation that
does not explicitly depend on the parameter~$\mu$, which can
nonetheless be recovered from~$u$. Making a guess of
what~$u$ should look like, we then write
\[
u(x)= u_0(x)+|x|v_0(x)\,,
\]
where $u_0(x)\sim \sqrt{\pi/8}|x|^{1/2}$ is an explicit, carefully chosen approximate solution of the equation
and the correction term~$v_0(x)$ should then be obtained via an inverse
function theorem on $L^\infty(\TT)$. Up to a technicality (namely, that $v_0(x)$
appears in this formula with a factor of~$|x|$ instead
of~$|x|^{1+\eta}$), this proves the easier part of
Theorem~\ref{mainthm}, namely, the asymptotic formula~\eqref{asymp}.

We should emphasize, however, that this description hides three key difficulties that make the proof much harder than it
looks. A first, fairly obvious one is that the argument boils down to
estimates on $L^\infty([-\pi,\pi])$ for the linear operator
\[
T_0 f(x):= \frac1{2|x| u_0(x)}\big[ L(|\cdot| f)(x)+ L(|\cdot|
f(-\,\cdot\,))(x)-2L(|\cdot| f)(0)\big]\,,
\]
whose kernel is rather difficult to control. Indeed, the convolution
kernel of the operator~$L$
acting on $L^\infty(\TT)$ that appears in the definition
of~$T_0$ has the rather awkward
expression~\cite{EW}
\begin{equation}\label{kernel1}
Lf(x)=\int_{-\pi}^\pi K(y)\, f(x-y)\, dy\,,\qquad
  K(x)=\sum_{n=1}^\infty\int_{(n-\frac12)\pi}^{n\pi}\frac{\cosh
  [s(\pi-|x|)]}{\pi \sinh(s\pi)}\, \bigg(\frac{|\tan s|}s\bigg)^{1/2}\, ds
\end{equation}
for $x\in(-\pi,\pi)$.

A second, less obvious difficulty is that the
operator norm of~$T_0$ turns out to be very slightly smaller than~1. Therefore, the bound for the norm of $(I-T_0)^{-1}$ that we
need in the argument is large, and this has the crucial consequence
that it becomes very hard to construct an approximate solution~$u_0$
such that the associated error $T_0u_0-\frac{u_0}{2|x|}$ is small enough
in~$L^\infty(\TT)$. 

Finally, the third difficulty is that, as the solution~$\vp$ is not
smooth at the origin, one cannot effectively use ordinary or
trigonometric polynomials to construct the approximate solution~$u_0$ (which would interact well with the operator $L$),
as is customary in computer-assisted proofs, and plain powers~$|x|^s$
cannot be used to approximate the $2\pi$-periodic function~$u$ properly either as
they do not glue well at $x=\pm\pi$ and do not have simple representations whenever $L$ acts on them. Instead, to
construct~$u_0$ we utilize information about the asymptotic behavior
of the solutions at~0 and carefully concoct a linear combination of
trigonometric polynomials and Clausen functions of different orders, defined as
\begin{align}
  \label{Clausen}
  C_z(x)= \sum_{n=1}^{\infty} \frac{\cos(n x)}{n^z}\,,\quad S_z(x)= \sum_{n=1}^{\infty} \frac{\sin(n x)}{n^z}\,.
\end{align}
Suitable estimates for Clausen functions are derived in
order to obtain the required uniform bounds for the approximate
solution. Two relevant additional remarks are that choosing $u_0$ just
from asymptotic information at~0 is not possible, as the approximation
that one obtains away from zero is poor, and that $u_0$ is a
combination of~20 different terms, so carrying out the estimates
without a computer seems unwise. The need of so many terms is due to
the almost non-invertibility of $(I - T_0)$, which results in the need
for a very accurate approximate solution that cannot be constructed
using just a few explicit terms.

It should be stressed that
the hardest part of Theorem~\ref{mainthm}, that is the proof of the
convexity of the solution, is considerably more technical but is based
on the same principles, suitably strengthened to control two
derivatives of the function~$u$.
This is ultimately accomplished by solving an extended
system of equations that is controlled by three linear operators: the
aforementioned~$T_0$ and two new, more complicated operators $T_1$ and~$T_2$
that involve up to two derivatives of the (extremely messy)
approximate solution~$u_0$. Just as before, one needs to invert
$I-T_i$ for $0\leq i\leq 2$ and the norms of the three operators $T_i$
are very close but strictly less than~1.

To offer some perspective as to why these factors conspire to make the
proof so demanding without getting bogged down in technicalities,
suffice it to say that this is the first computer-assisted proof of
the existence of truly low-regularity (e.g., continuous but not~$C^1$)
solutions of any (ordinary or partial, even local) differential
equation. See
\cite{Castelli-Gameiro-Lessard:rigorous-numerics-ill-posed-pde,DeLaLlave-Sire:a-posteriori-kam-ill-posed-pde} for computer-assisted proofs of periodic solutions or KAM tori of ill-posed
PDE.

A major theme of our work is the interplay between rigorous
computer calculations and traditional mathematics; throughout the
paper we use interval arithmetics as part of a proof whenever they are needed.
Lately, computer-assisted proofs have been made possible due to the increment of computational resources.  Naturally, floating-point operations can result in numerical errors. In order to overcome these, we will employ interval arithmetics to deal with this issue. The main paradigm is the following: instead of working with arbitrary real numbers, we perform computations over intervals which have representable numbers by the computer as endpoints in order to guarantee that the true result at any point belongs to the interval by which is represented. On these objects, an arithmetic is defined in such a way that we are guaranteed that for every $x \in X, y \in Y$
\begin{align*}
x \star y \in X \star Y,
\end{align*}
for any operation $\star$. For example,
\begin{align*}
[\underline{x},\overline{x}] + [\underline{y},\overline{y}] & = [\underline{x} + \underline{y}, \overline{x} + \overline{y}] \\
[\underline{x},\overline{x}] \times [\underline{y},\overline{y}] & = [\min\{\underline{x}\underline{y},\underline{x}\overline{y},\overline{x}\underline{y},\overline{x}\overline{y}\},\max\{\underline{x}\underline{y},\underline{x}\overline{y},\overline{x}\underline{y},\overline{x}\overline{y}\}].
\end{align*}
We can also define the interval version of a function $f(X)$ as an interval $I$ that satisfies that for every $x \in X$, $f(x) \in I$. Rigorous computation of integrals has been theoretically developed since the seminal works of Moore and many others \cite{Berz-Makino:high-dimensional-quadrature,Castro-Cordoba-GomezSerrano:global-smooth-solutions-sqg,Cordoba-GomezSerrano-Zlatos:stability-shifting-muskat-II,Kramer-Wedner:adaptive-gauss-legendre-verified-computation,Lang:multidimensional-verified-gaussian-quadrature}. We also refer the reader to the books \cite{Moore-Bierbaum:methods-applications-interval-analysis,Tucker:validated-numerics-book} and to the survey \cite{GomezSerrano:survey-cap-in-pde} for a more specific treatment of computer-assisted proofs in PDE.

At this stage, it is worth mentioning why the strategy
of the proof of Theorem~\ref{mainthm} is so different from the
celebrated proof of the convexity of the highest Stokes waves. In
short, the reason is that, although Nekrasov's equation for the
interface reduces the problem to the analysis of a nonlinear, nonlocal equation
similar to Whitham's, the actual proof of the conjecture due to
Plotnikov and Toland~\cite{PT} hinges on the equivalent formulation of the
problem in terms of a harmonic function on the half-strip
$(-\infty,0)\times (-\pi,\pi)$ satisfying certain Dirichlet and
Neumann boundary conditions. In a tour de force of complex analysis,
this overdetermined boundary condition is shown to imply that the
boundary conditions of the above harmonic function can be written in
terms of a holomorphic function satisfying a certain ODE in the
complex plane. Via the Poisson kernel, this leads to writing the
derivative of the function~$\te(x)$ that describes the water interface
as
\[
\te'(x)=\int_0^\infty \frac{\Phi'(y)\,\sinh y}{\cosh y-\cos x}\, dy\,,
\]
with $\Phi'(y)>0$. Hence $\te'(x)>0$, and this automatically implies
that the interface is convex. In our case, the problem does not admit
a local description and is not amenable to the use of complex-analytic methods, so one needs to
work directly with the Whitham equation using real-variable
techniques.

Of course, the proof of Theorem~\ref{mainthm} would be much easier if
one could come up with a simpler strategy where soft analysis could be
used to bypass the need for hard estimates, but it is hard to imagine
what such a strategy could be based on. Let us briefly comment on this
important point. For instance, a first idea would be to try to adapt
the proof of Ehrnstr\"om and Wahl\'en to include as part of the
functional space the additional fact that the functions are
convex. However, some work shows that this philosophy cannot be easily
implemented since it is by no means clear how to carry out the local
or global bifurcation argument within this framework. Another obvious
idea is to carry out the global bifurcation argument directly in a
$C^{1/2}$ H\"older space. Alas, showing that $C^{1/2}$~is indeed the
sharp H\"older regularity of the resulting solution, meaning that it
does not belong to a higher space in the H\"older scale, turns out to
be highly nontrivial in this approach.

The paper is organized as follows.
In Section~\ref{S.clausen} we give some technical results concerning
generalized Clausen functions and their asymptotic behavior at $x=0$.
Equipped with these formulas, in Section~\ref{S.approx.sol.}
we construct an approximate solution~\eqref{u0} to the equation verified
by~\eqref{udef}, what we call the reduced Whitham
equation~\eqref{reducedW},
and whose error with respect to an exact
solution is small (when suitably measured in $L^\infty$)
for our purposes.
In this section we also record all the estimates involving the
approximate solution and its derivatives close to $x=0$ that we need
later on the paper.

A linearized version of the Whitham equation is studied in
Section~\ref{S.lineareq}.
Here we use a fixed point argument together
with the invertibility of $1-T_0$ to show
the existence of a solution which is an $L^\infty$ small perturbation
of our approximate solution that displays (almost) the right
asymptotic behavior claimed in Conjecture~\ref{conj}.

Section~\ref{S.convexity} is devoted to the proof of the main
Theorem~\ref{mainthm}.
We exploit the bounds for the norms of linear operators $T_0,T_1$ and
$T_2$ to set up a fixed point scheme that allows us to
conclude the convexity of the highest cusped Whitham wave.
A precise control of the error terms as well as the computation of
some explicit constants that appear along the argument is done using computer assistance.

Two appendices are given the end of the paper.
For convenience, we leave the study of the norm of $T_2$ for
Appendix~\ref{S.AppendixA} meanwhile in Appendix~\ref{S.AppendixB}, we
give the details of the computer assisted proof and numeric
computations that we use throughout the paper.








\section{Some technical lemmas about Clausen
  functions}{\label{S.clausen}}

For the benefit of the reader,
in this section we provide all the estimates for the generalized Clausen
functions introduced in~\eqref{Clausen} that we use in the rest
of the paper.
In particular, as mentioned in the introduction, these
functions play a fundamental role in the proof of
Theorem~\ref{mainthm} as they are the building blocks of the
approximate solution that we shall present in the next section.

Let us begin with the relationship between
Clausen functions and the polylogarithm
$\Li_z(s)$~\cite[Eq.~25.12.10]{DLMF}:
\begin{equation}\label{polylog}
  \Li_z(s)=\sum_{n=1}^\infty \frac{s^{n}}{n^z}.
\end{equation}
This series defines an analytic function for all complex $z$ whenever $|s|<1$
and it can be analytically continued for other values.
Further, recalling the definition of Clausen functions~\eqref{Clausen},
it is clear now that
\begin{align*}
  C_z(x)&=\frac 1 2 \big( \Li_z(e^{ix})+ \Li_z(e^{-ix}) \big)
          =\operatorname{Re}\big(\Li_z(e^{i x})\big)\,,\\
  S_z(x)&=\frac{1}{2i} \big( \Li_z(e^{ix})-\Li_z(e^{-ix}) \big)
          =\operatorname{Im}\big(\Li_z(e^{i x})\big)\,.
\end{align*}
By the well known identity~\cite[Eq.~25.12.12]{DLMF},
\[
\Li_z(s)=\Ga(1-z)\big( \log(s^{-1})
\big)^{z-1}+\sum_{n=0}^{\infty}\ze(z-n)\frac{(\log s)^n}{n!}\,,\quad
z\notin\Z_+\,, |\log(s)|<2\pi\,.
\]
one has the following series representations for $C_z$ and
$S_z$:
\begin{align}
  \label{ClausenC}
  C_z(x)&=\Ga(1-z) \sin(\tfrac \pi 2 z)|x|^{z-1}
          +\sum_{m=0}^\infty(-1)^m\ze(z-2m)\frac{x^{2m}}{(2m)!}\\
      \label{ClausenS}
  S_z(x)&=\Ga(1-z) \cos(\tfrac \pi 2 z) \sign(x)|x|^{z-1}
  +\sum_{m=0}^\infty(-1)^m\ze(z-2m-1)\frac{x^{2m+1}}{(2m+1)!}\,,
\end{align}
where $\ze(z)$ is the Riemann zeta function.
Observe that these formulas (analytically) extend the
definition~\eqref{Clausen} when $\operatorname{Re}(z)<1$ for all $x$
real.

As it will be useful later on,
in the following lemma we give uniform bounds for the lower order
terms in the above series:

\begin{lemma}\label{boundsClausen0}
  Let $z$ be a positive real number and let $M= \ceil[\Big]{\frac{z+1}{2}}$.
  Then the Clausen functions can be expressed as follows:
  \begin{align*}
    C_z(x)&= \Ga(1-z) \sin(\tfrac \pi 2 z)|x|^{z-1}+\ze(z)
          +\sum_{m=1}^{M-1}(-1)^m\ze(z-2m)\frac{x^{2m}}{(2m)!}+E_{C_z}(x)\\
  S_z(x)&= \Ga(1-z) \cos(\tfrac \pi 2 z) \sign(x)|x|^{z-1}
          +\sum_{m=0}^{M-1}(-1)^m\zeta(z-2m-1)\frac{x^{2m+1}}{(2m+1)!}
          +E_{S_z}(x)\,,
  \end{align*}
where the error terms
  \begin{align}
    \label{errorClausen}
    | E_{C_z}(x)|&\leq 2 (2\pi)^{1+z-2M}\ze(2M+1-z)
                   \frac{x^{2M}}{4\pi^2-x^2}\,,\\
     | E_{S_z}(x)|&\leq  2 (2\pi)^{z-2M}\ze(2M+2-z)
                    \frac{|x|^{2M+1}}{4\pi^2-x^2}\,. 
  \end{align}
  
\end{lemma}

\begin{proof}
As they are similar, for simplicity we only
prove the estimate for $C_z(x)$.
From \eqref{ClausenC} we have that for any positive integer $M$,
\begin{align*}
C_z(x)&= \Ga(1-z) \sin(\tfrac \pi 2 z)|x|^{z-1}+\ze(z)
+\sum_{m=1}^{M-1}(-1)^m\ze(z-2m)\frac{x^{2m}}{(2m)!}\\
&\qquad \qquad \qquad
   +\sin(\tfrac \pi 2 z)\sum_{m=M}^{\infty}\frac{\Ga(2m+1-z)}{\Ga(2m+1)}
            2^{z-2m}\pi^{z-2m-1}\ze(2m+1-z) x^{2m}\,,
\end{align*}
where we have used that $\Ga(2m+1)=(2m)!$ and the functional identity
\begin{equation}
  \label{identz}
  \ze(s)=2^{s}\pi^{s-1}\sin(\tfrac \pi 2 s)\Ga(1-s)\ze(1-s)\,,
\end{equation}
which is valid for all $s\in\C$.

Since $\ze(s)$ and $\Ga(s)$ are, respectively,
monotonically decreasing and
increasing functions on $s>2$,
by taking $M= \ceil[\Big]{\frac{z+1}{2}}$,
we arrive at
\begin{align*}
|E_{C_z}(x)|&\leq  \ze(2M+1-z) |\sin(\tfrac \pi 2 z)| \sum_{m=M}^{\infty}2^{z-2m}\pi^{z-2m-1}x^{2m}\,.
\end{align*}
Noticing that the above infinite sum converges for $|x|\leq \pi$,
\[
  \sum_{m=M}^{\infty}2^{z-2m}\pi^{z-2m-1}x^{2m}=
2  (2\pi)^{1+z-2M}\frac{ x^{2M}}{4\pi^2-x^2}\,,
\]
the estimate for $C_z$ follows.
\end{proof}

\section{Approximate solution}\label{S.approx.sol.}

Our aim here is to introduce an approximate solution $u_0$ to
the Whitham equation satisfied by~\eqref{udef}
and study its asymptotic behavior at $x=0$.
Making use of the estimates derived in the previous section,
we will be able to control the $L^\infty$-norm of the error made with
respect to an exact solution and prove
that it is sufficiently small for the purposes of the fixed point
iteration scheme that we set up later in the paper.


Let us begin by introducing the linear operator
\begin{equation}
  \label{L}
\L u= \frac12 \int_{-\pi}^{\pi } \big( K(x-y)+K(x+y)-2K(y) \big) u(y)\,,
\end{equation}
with $K(x)$ as in~\eqref{kernel1}.
Furthermore, it will be useful to express the kernel $K$ alternatively
as the Fourier series~\cite{EW}
\begin{equation}
  \label{kernel2}
  K(x)=\frac{1}{2\pi}\sum_{n\in \Z} m(n) e^{i nx}=\frac{1}{\pi}\sum_{n=1}^{\infty} m(n) \cos(n x)\,,
\end{equation}
where in the second equality we have used the parity of $m(n)=\sqrt{\dfrac{\tanh(n)}{n}}$.
\begin{remark}\label{R.K}
  Notice that by the representation of the Clausen
  function $C_{1/2}$ given in Lemma~\ref{boundsClausen0},
  the above kernel satisfies
\[
 K(x)=\frac{1}{\sqrt{2\pi |x|}}+E_{\mathrm{reg}}(x)\,,\qquad
 E_{\mathrm{reg}}(x)=
 E_{C_{\frac12}}(x)+\frac1\pi\sum_{n=1}^\infty
 \frac{1-\sqrt{\tanh(n)}}{\sqrt{n}}\cos(nx) \,,
\]
which agrees with the description given in \cite[Prop. 3.1]{EW}.
\end{remark}

Through the paper we will take advantage of the fact that
$u$ defined as in~\eqref{udef} satisfies
a quadratic equation that does not depend explicitly on the parameter $\mu$:
\begin{prop}\label{asympu0}
Let $\vp(x)$ be a solution of \eqref{Wphi}.
Then, the function $u(x)=\frac \mu 2 -\vp(x)$ satisfies the reduced Whitham equation
\begin{equation}
  \label{reducedW}
  (u(x))^2=\L u(x)\,,
\end{equation}
where the wave-speed $\mu$ is recovered through
\begin{equation}
  \label{mueq}
  \mu \Big( 1-\frac{\mu}{ 2} \Big)=4\int_{0}^{\pi} K(y) u(y)\,dy\,.
\end{equation}
\end{prop}
\begin{remark}\label{R.mu}
  Notice that in view of the Galilean transformation
  \[
    \mu\mapsto 2-\mu\,,\quad \vp\mapsto \vp+1-\mu\,,
  \]
  solutions $\vp$ to~\eqref{Wphi}
  with wave-speed $\mu\in [1,2]$ are mapped bijectively to
  solutions for $\mu\in [0,1]$ with maxima in $[0,\mu/2]$.
  Since $K$ and $u$ are positive by~\eqref{kernel1} and the fact that
  $\vp<\frac\mu 2$, the quadratic equation~\eqref{mueq} for $\mu$
  has only one root in $[0,1]$, which is precisely the value of $\mu$
  associated to the highest wave $\vp$.
\end{remark}

As we will show in the next section, \eqref{reducedW} imposes
strong restrictions on the asymptotic behavior of the solution $u$.
In particular, since Remark~\ref{R.K} entails that
\begin{equation*}
  \int_0^\pi \big(C_{\frac 12}(x-y)+C_{\frac 12} (x+y)-2 C_{\frac 12}(y)\big)\sqrt{y}\,dy
  =\frac \pi 2 |x|+O(x^2)\,,
\end{equation*}
by using the formula \eqref{ClausenC} one can easily show the
following asymptotic formula:
\begin{prop}\label{L.lambda}
Let $\la>0$ and assume that $u$ is a solution of \eqref{reducedW} with the asymptotic behavior
\[
u(x)=\la \sqrt{|x|}+O( |x|^{\frac 12 +p})\,,\quad
p>0\,,
\]
close to $|x|=0$.
Then the constant $\la$ must take the value
\begin{equation}
  \la=\sqrt{\frac{\pi}{8}}\,.\label{lambda}
\end{equation}
\end{prop}

An approximate solution to the reduced Witham equation is given
now in terms of Clausen functions and trigonometric polynomials.
We postpone to the subsequent section the construction of an actual
solution of~\eqref{reducedW} with the desired behavior at $x=0$.
\begin{defi}
  Let $p_0\,, p_1$ be positive numbers given by
\begin{equation}
  \label{exponents}
  \frac{\Ga(-1/2-p_j)}{\Ga(-1-p_j)}
  \big(1-\cot(\tfrac \pi 2 p_j) \big)=\frac{2}{\sqrt{\pi}}\,,
\end{equation}
with $0<p_0<1$ and $2<p_1<3$. Then, we define
  \begin{equation}
  \label{u0}
  u_0(x)= \sum_{k=0}^{N_j} \sum_{j=0}^{1} a_{jk}
          \big(\zeta\big(3/2+kp_0+jp_1\big)-C_{\frac 3 2
          +kp_0+jp_1}(x)\big)+\sum_{n=1}^{N_2} b_{n}
          \big(\cos(n x)-1\big)\,,
\end{equation}
in which the coefficients $a_{jk}$ and $b_n$ are real and
$N_{0}, N_1 $ and $N_2$ are fixed non-negative integers.
\end{defi}
In view of this definition and the formulas~\eqref{ClausenC} and
\eqref{ClausenS}, we have in addition that the derivatives of $u_0$ can
be written as
\begin{align}
  \label{u0'}
  u_0'(x)&= \sum_{k=0}^{N_j} \sum_{j=0}^{1} a_{jk}
        S_{\frac 1 2 + kp_0+jp_1}(x)-\sum_{n=1}^{N_2} n b_{n} \sin(n x)\,,
\end{align}

\begin{align}
  \label{u0''}
  u_0''(x)&= \sum_{k=0}^{N_j} \sum_{j=0}^{1} a_{jk}
    C_{-\frac 1 2 + kp_0+jp_1}(x)-\sum_{n=1}^{N_2} n^2 b_{n} \cos(n x)\,.
\end{align}
Moreover, the explicit values of the numbers $p_0,p_1$ can be enclosed
with high precision
\begin{lemma}\label{L.p0p1}
  The solutions $p_0,p_1$ of the equation~\eqref{exponents} are
  \[
    p_0=0.611\ldots\,,\quad p_1=2.762\ldots
   \]
\end{lemma}

A key feature of the approximate solution $u_0$ is precisely its asymptotic
behavior near $x=0$. In fact, the bounds shown in
Lemma~\ref{boundsClausen0} imply the following asymptotic expansions
that we give without proof as they involve tedious but largely
standard computations: 
\begin{lemma}\label{L.estu0}
  Let $u_0$ be a function of the form \eqref{u0} and let
  $M$ be the smallest integer such that
  $M\geq  3/2+\max\{N_0p_0, N_1 p_0+p_1\}$.
  Then, the following asymptotic expansions hold near
  $x=0$:
\begin{multline}
    \label{asympu0}
  u_0(x)=\sum_{k=0}^{N_j}\sum_{j=0}^1  a^{0}_{jk}
  |x|^{\frac 12+kp_0+jp_1}
  +\Big(a^0_1-\frac12\sum_{n=1}^{N_2}n^2b_n\Big)x^2
+ \Big(a^0_2+\frac{1}{24} \sum_{n=1}^{N_2}n^4b_n\Big)x^4
   +E_{u_0}(x)\,,
 \end{multline}
 \begin{multline}
        \label{asympu1}
 u'_0(x)=\sum_{k=0}^{N_j}\sum_{j=0}^1  a^{1}_{jk} |x|^{-\frac12+kp_0+jp_1}
   +\Big(a^1_0-\sum_{n=1}^{N_2}n^2b_n\Big)|x|
+ \Big(a^1_1+\frac{1}{6} \sum_{n=1}^{N_2}n^4b_n\Big)|x|^3
+E_{u'_0}(x)\,,
\end{multline}
and
 \begin{multline}
            \label{asympu2}
u''_0(x)=\sum_{k=0}^{N_j}\sum_{j=0}^1  a^{2}_{jk} |x|^{-\frac 32+kp_0+jp_1}
+\Big(a^2_1-\sum_{n=1}^{N_2}n^2b_n\Big)
+ \Big(a^2_2+\frac{1}{2} \sum_{n=1}^{N_2}n^4b_n\Big)x^2
    +E_{u''_0}(x)\,,
\end{multline}
  where
\begin{align*}
   a^0_{jk} &=-\Ga(-1/2 -kp_0-jp_1)
              \sin\big(\tfrac \pi2 (\tfrac 32 +k p_0+jp_1)\big) a_{jk}\,,\\
  a^0_m&=\frac{(-1)^{m+1}}{(2m)!}\sum_{k=0}^{N_j}\sum_{j=0}^1 a_{jk}
              \ze(3/2+kp_0+jp_1-2m)\,,\quad m=1,2\,,
  \\
  a^1_{jk} &=\Ga(1/2 -kp_0-jp_1)
             \cos\big(\tfrac \pi2 (\tfrac 12 +k p_0+jp_1)\big)
             a_{jk}\,,\\
  a^1_m&=\frac{(-1)^{m}}{(2m+1)!}\sum_{k=0}^{N_j}\sum_{j=0}^1 a_{jk}
              \ze(-1/2+kp_0+jp_1-2m)\,, \quad m=0,1\,,\\
    a^2_{jk} &=-\Ga(3/2 -kp_0-jp_1)
               \sin\big(\tfrac \pi2 (-\tfrac 12 +k p_0+jp_1)\big) a_{jk}\\
    a^2_m&=\frac{(-1)^{m}}{(2m)!}\sum_{k=0}^{N_j}\sum_{j=0}^1 a_{jk}
              \ze(-1/2+kp_0+jp_1-2m)\,, \quad m=1,2\,,
\end{align*}
and
\begin{multline}\label{erroru0}
|E_{u_0}(x)|\leq
2 (2\pi)^{5/2-2M}\sum_{k=0}^{N_j}\sum_{j=0}^1
(2\pi)^{kp_0+jp_1} |\ze(2M-1/2-kp_0-jp_1)a_{jk}  |
\frac{x^{2M}}{4\pi^2-x^2}\\
+\frac{x^6}{6!}\sum_{n=1}^{N_2}n^6|b_n|+
\sum_{m=3}^{M-1}|a^0_m| x^{2m}
\,,
\end{multline}

\begin{multline}\label{erroru0'}
|E_{u'_0}(x)|\leq
2 (2\pi)^{3/2-2M}\sum_{k=0}^{N_j}\sum_{j=0}^1
(2\pi)^{kp_0+jp_1} |\ze(2M+3/2-kp_0-jp_1)a_{jk}  | \frac{|x|^{2M+1}}{4\pi^2-x^2}\\
+\frac{|x|^5}{5!}\sum_{n=1}^{N_2}n^6|b_n|+
\sum_{m=2}^{M-1}|a^1_m| |x|^{2m+1}\,.
\end{multline}

\begin{multline}
|E_{u''_0}(x)|\leq
2 (2\pi)^{1/2-2M}\sum_{k=0}^{N_j}\sum_{j=0}^1
(2\pi)^{kp_0+jp_1} |\ze(2M+5/2-kp_0-jp_1)a_{jk}  | \frac{x^{2M}}{4\pi^2-x^2}\\
+\frac{x^4}{4!}\sum_{n=1}^{N_2}n^6|b_n|+
\sum_{m=2}^{M-1}|a^2_m| x^{2m}\,.
\end{multline}
Moreover, the derivatives of the error term $E_{u_0}(x)$
are trivially bounded as
\begin{equation*}
  |E_{u_0}'(x)|\leq |E_{u'_0}(x)|\,,\quad
   |E_{u_0}''(x)|\leq |E_{u''_0}(x)|\,. 
\end{equation*}
\end{lemma}
Analogously, will also need estimates for $\L u_0$ with $\L$ the linear
operator introduced in~\eqref{L} (that we give without proof, as
before):
\begin{lemma}\label{L.estLu0}
The asymptotic expansion of $\L u_0$ close $x=0$ is
  \begin{multline}
    \label{lu0}
    \L u_0=\sum_{k=0}^{N_j}\sum_{j=0}^1  A^{0}_{jk} |x|^{1+kp_0+jp_1}\\
    +\Big(A^0_1-\frac12\sum_{n=1}^{N_2}b_n n^{3/2}
    \sqrt{\tanh(n)}
+\frac12\sum_{n=1}^{\infty} \sum_{k=0}^{N_j}\sum_{j=0}^1 a_{jk}
\frac{\sqrt{\tanh(n)}}{n^{kp_0+jp_1}}
    \Big)x^2\\
+\Big(A^0_2+\frac{1}{24} \sum_{n=1}^{N_2}b_n n^{7/2}
\sqrt{\tanh(n)}
-\frac{1}{24}\sum_{n=1}^{\infty} \sum_{k=0}^{N_j}\sum_{j=0}^1 a_{jk}
\frac{\sqrt{\tanh(n)}}{n^{kp_0+jp_1-2}}
\Big)x^4
    +E_{\L u_0}(x)\,,
  \end{multline}
  where
  \begin{align*}
  A^0_{jk}&=\Ga(-1-kp_0-jp_1) \sin\big(\tfrac \pi2
    (kp_0+jp_1)\big)a_{jk}\,.\\
A^0_m&=\frac{(-1)^{m+1}}{(2m)!}\sum_{k=0}^{N_j}\sum_{j=0}^1 a_{jk}
              \ze(2+kp_0+jp_1-2m)\,,
  \end{align*}
  and
  \begin{multline*}
    |E_{\L u_0}(x)|\leq
    2 (2\pi)^{3-2M}\sum_{k=0}^{N_j}\sum_{j=0}^1
(2\pi)^{kp_0+jp_1} |\ze(2M-1-kp_0-jp_1) a_{jk}  |
\frac{x^{2M}}{4\pi^2-x^2}\\
+\frac{x^6}{6!}\sum_{n=1}^{N_2}n^{3/2}\sqrt{\tanh(n)}|b_n|
+\sum_{m=3}^{M-1}|A^0_m| x^{2m}\,.
\end{multline*}
Moreover, the derivatives of the error term have the following bounds:
  \begin{multline*}
    |E_{\L u_0}'(x)|\leq
    2 (2\pi)^{2-2M}\sum_{k=0}^{N_j}\sum_{j=0}^1
    (2\pi)^{kp_0+jp_1} |\ze(2M+1+kp_0-jp_1)a_{jk}  |
\frac{|x|^{2M+1}}{4\pi^2-x^2}\\
+\frac{|x|^5}{5!}\sum_{n=1}^{N_2}n^{3/2}\sqrt{\tanh(n)}|b_n|
+\sum_{m=2}^{M-1}|A^1_m| |x|^{2m+1}\,,
\end{multline*}
  \begin{multline*}
    |E_{\L u_0}''(x)|\leq
2 (2\pi)^{1-2M}\sum_{k=0}^{N_j}\sum_{j=0}^1
(2\pi)^{kp_0+jp_1} |\ze(2M+3+kp_0-jp_1) a_{jk} |
\frac{x^{2M}}{4\pi^2-x^2}\\
+\frac{x^4}{4!}\sum_{n=1}^{N_2}n^{3/2}\sqrt{\tanh(n)}|b_n|
+\sum_{m=3}^{M-1}|A^2_m| x^{2m}\,,
\end{multline*}
with
\begin{align*}
A^1_m&=\frac{(-1)^{m}}{(2m+1)!}\sum_{k=0}^{N_j}\sum_{j=0}^1 a_{jk}
              \ze(kp_0+jp_1-2m)\,,\\
A^2_m&=\frac{(-1)^{m}}{(2m)!}\sum_{k=0}^{N_j}\sum_{j=0}^1 a_{jk}
              \ze(kp_0+jp_1-2m)\,.
\end{align*}
\end{lemma}

At this point we can now understand the 
construction of the approximate solution $u_0$
and how it helps us to address Conjecture~\ref{conj}.
Indeed, let $u(x)=u_0(x)+|x| v_0(x)$ be a solution of
the reduced Whitham equation \eqref{reducedW}
with $u_0(x)$ as before and $v_0(x)\in L^\infty(\T)$.
In terms of the perturbation $v_0$
the equation can be recast as
\begin{equation}
  \label{v0}
  (I-T_0) v_0=F_0-\frac{|x|}{2u_0} v_0^2\,,
\end{equation}
with $T_0:L^{\infty}(\T)\to L^{\infty}(\T)$ the operator
\begin{equation}
  \label{T0}
  T_0 v_0(x)=\frac{1}{2|x|u_0}\int_{0}^{\pi} \big(
  K(x-y)+K(x+y)-2K(y)\big) y\, v_0(y)\,dy\,,
\end{equation}
and where we have defined
\begin{equation}
  \label{F0}
  F_0=\frac{1}{2|x| u_0}(\L u_0-u_0^2)\,.
\end{equation}
Since we aim to show the existence of such $v_0$ in $L^\infty(\T)$,
the precise form of $u_0$ in~\eqref{u0} becomes apparent:
we choose the coefficients $a_{jk}$ so that the defect term $F_0$ is
bounded in $L^\infty(\T)$ and arbitrarily small close to $x=0$ while
the constants $b_n$ are chosen to control the norm globally.

For notational convenience and before we give a uniform bound for
$F_0$, let us introduce an auxiliary function
\begin{equation}\label{hatu0}
\hat u_0(x):=\frac{\la\sqrt{x}-u_0(x)}{u_0(x)}\,,\quad x\in [0,\pi]\,,
\end{equation}
which is small close to $x=0$ as the following lemma shows:
\begin{lemma}\label{L.hatu0}
  Let $\hat u_0$ be as before and take $\ep>0$ a small fixed
  number. Then, for $0\leq x\leq \ep$,
\[
\hat u_0(x)\leq  c_{\ep,\hat u_0} x^{p_0}\,.
\]
\end{lemma}
\begin{proof}
  By the definition of $\hat u_0(x)$,
  \begin{multline*}
    \hat u_0(x)\leq \Big(
      \sum_{j+k>0} |a^{0}_{jk}| |x|^{(k-1)p_0+jp_1}
  +\Big|a^0_1-\frac12\sum_{n=1}^{N_2}n^2b_n\Big||x|^{\frac32-p_0}
  +\Big|a^0_2+\frac{1}{24} \sum_{n=1}^{N_2}n^4b_n\Big| |x|^{\frac72-p_0}\\
  +|x|^{-1/2}|E_{u_0}(x)|\Big)
\Big(
  \la- \sum_{j+k>0} |a^{0}_{jk}||x|^{kp_0+jp_1}
  -\Big|a^0_1-\frac12\sum_{n=1}^{N_2}n^2b_n\Big||x|^{3/2}\\
  -\Big|a^0_2+\frac{1}{24} \sum_{n=1}^{N_2}n^4b_n\Big| |x|^{7/2}
  -|x|^{-1/2}|E_{u_0}(x)|\Big)^{-1}|x|^{p_0}\,.
\end{multline*}
Using the monotonicity of all the terms in the above expression,
by evaluating the fraction at $|x|=\ep$ we obtain the constant
$c_{\ep,\hat u_0}$:
\begin{multline*}
c_{\ep,\hat u_0}:=\Big(
      \sum_{j+k>0} |a^{0}_{jk}| \ep^{(k-1)p_0+jp_1}
  +\Big|a^0_1-\frac12\sum_{n=1}^{N_2}n^2b_n\Big|\ep^{\frac32-p_0}
  +\Big|a^0_2+\frac{1}{24} \sum_{n=1}^{N_2}n^4b_n\Big| \ep^{\frac72-p_0}+\ep^{-1/2} E_{\ep,u_0}|\Big)\\
\Big(
  \la- \sum_{j+k>0} |a^{0}_{jk}|\ep^{kp_0+jp_1}
  -\Big|a^0_1-\frac12\sum_{n=1}^{N_2}n^2b_n\Big|\ep^{3/2}\\
  -\Big|a^0_2+\frac{1}{24} \sum_{n=1}^{N_2}n^4b_n\Big| \ep^{7/2}
  -|x|^{-1/2} E_{\ep,u_0}|\Big)^{-1}\,,
\end{multline*}
where
\begin{multline*}\label{Eepu0}
E_{\ep,u_0}=2 (2\pi)^{5/2-2M}\sum_{k=0}^{N_j}\sum_{j=0}^1
(2\pi)^{kp_0+jp_1} |\ze(2M-1/2-kp_0-jp_1)a_{jk}  |
\frac{\ep^{2M}}{4\pi^2-\ep^2}\\
+\frac{\ep^6}{6!}\sum_{n=1}^{N_2}n^6|b_n|+
\sum_{m=3}^{M-1}|a^0_m| \ep^{2m}
\end{multline*}
denotes the RHS of~\eqref{erroru0} at $x=\ep$.
\end{proof}

\begin{lemma}\label{L.F0}
  Let $u_0$ be as in~\eqref{u0} with $a_{00}=\frac{1}{4}$.
  Then $F_0\in L^\infty(\T)$
  and
  \begin{equation}
    \label{de0}
    \de_0:=\|F_0\|_{L^\infty(\T)}\leq 9.1\cdot 10^{-8}\,.
  \end{equation}
  
\end{lemma}
\begin{proof}
  A long but straightforward computation shows that
  \begin{multline}\label{Lu0-u02}
  \L u_0(x)-u_0^2(x)=\big(A^0_{00}-(a^0_{00})^2\big)|x|
  +(A^0_{01}-2a^0_{00}a^0_{01})|x|^{1+p_0}
  +(A^0_{10}-2a^0_{00}a^0_{10})|x|^{1+p_1}\\
+( A_{11}^0-a^0_{00}a^0_{11}-a^0_{01}a^0_{10})|x|^{1+p_0+p_1}
  \\
  +\sum_{k=2}^{N_0}\Big( A_{0k}^0-\frac12 ((-1)^k+1)
  (a_{0\floor{\frac k 2}}^0)^2
  -2\sum_{j=0}^{\floor{\frac{k-1}{2}}}a_{0j}^0a_{0 (k-j)}^0\Big)
  |x|^{1+k p_0}\\
  +\Big[
  A^0_{1}-\frac{1}{2}\Big( \sum_{n=1}^{N_2}b_n n^{3/2}
\sqrt{\tanh(n)}-
\sum_{n=1}^{\infty} \sum_{k=0}^{N_j}\sum_{j=0}^1 a_{jk}
\frac{\sqrt{\tanh(n)}}{n^{kp_0+jp_1}}\Big)\Big]x^2\\
+\Big[
  A^0_{2}+\frac{1}{24}\Big( \sum_{n=1}^{N_2}b_n n^{7/2}
\sqrt{\tanh(n)}-
\sum_{n=1}^{\infty} \sum_{k=0}^{N_j}\sum_{j=0}^1 a_{jk}
\frac{\sqrt{\tanh(n)}}{n^{kp_0+jp_1-2}}\Big)\\
-\Big(a^0_1-\frac12\sum_{n=1}^{N_2}n^2b_n\Big)^2
\Big]x^4-\Big(a^0_2+\frac{1}{24}
\sum_{n=1}^{N_2}n^4b_n\Big)^2x^8\\
-2\Big(a^0_1-\frac12\sum_{n=1}^{N_2}n^2b_n\Big) \sum_{k=0}^{N_j}\sum_{j=0}^1  a^{0}_{jk}
|x|^{\frac 52+kp_0+jp_1}
-2\Big(a^0_1-\frac{1}{24}\sum_{n=1}^{N_2}n^4b_n\Big) \sum_{k=0}^{N_j}\sum_{j=0}^1  a^{0}_{jk}
|x|^{\frac 92+kp_0+jp_1}\\
-2E_{u_0}(x) \sum_{k=0}^{N_j}\sum_{j=0}^1  a^{0}_{jk}
  |x|^{\frac 12+kp_0+jp_1}+E_{\L_{u_0}}(x)-(E_{u_0}(x))^2\,.
\end{multline}

Using now Lemma~\ref{L.hatu0} to write
\begin{equation*}
  \frac{1}{u_0(x)}=\frac{1+\hat u_0(x)}{\la \sqrt{x}}\,,
\end{equation*}
by~\eqref{F0} the coefficient $A^0_{00}-(a^0_{00})^2$ must then
vanish identically to ensure that $F_0\in L^\infty(\T)$, which holds
to be true when $a_{00}$ takes the value $\frac14$ by Lemma~\ref{L.estu0}.
The rest of the proof is computer assisted. See Appendix~\ref{S.AppendixB}.
\end{proof}
\begin{remark}
  Notice that $a_{00}=\frac14$ is equivalent to fix
  $a^{0}_{00}=\la$ in~\eqref{asympu0}, with $\la$ the constant
  of~\eqref{lambda}.
  Since we express a solution of~\eqref{reducedW} as
  $u(x)=u_0(x)+|x|v_0(x)$ for some $v_0\in L^\infty(\T)$,
  this condition is naturally expected
  by Proposition~\ref{L.lambda}
\end{remark}


%

\begin{lemma}\label{L.upperbounds}
  Let $u_0$ be the approximate solution~\eqref{u0} and take $\ep=0.1$.
  Then, 
  the following inequalities hold for $0\leq x\leq \ep$:
  \begin{align}
   \frac{ 1}{\la\sqrt{x}}(\la\sqrt{x}-u_0(x))&\leq
    \frac{1}{\la}c_{\ep,p_0} x^{p_0}\,,\label{L.upperbound0}\\
    \frac{1}{2x}-\frac{u_0'(x)}{u_0(x)}
    &\leq \frac{1}{\la} c_{\ep,p_0}' x^{p_0-1}\,, \label{L.upperbound1}\\
      \frac{3}{4 x^2}
     -\frac{1}{(u_0(x))^2}\big(2(u'_0(x))^2-u_0(x)u_0''(x)\big)
&\leq  \frac{c_{\ep,p_0}'' }{\la} x^{p_0-2} \,, \label{L.upperbound2}
   \end{align}
   where the constants $c_{\ep, p_0}, c_{\ep,p_0}', c_{\ep,p_0}''$ verify
   \begin{equation}
     c_{\ep,p_0}<0.142\,,\quad c_{\ep,p_0}'<0.16\,,\quad c_{\ep, p_0}''<0.178\,.
   \end{equation}
\end{lemma}
\begin{proof}
We only show the first two bounds as the third is obtained in the same way.
To start, observe that by~\eqref{hatu0},
\[
\frac{1}{2x}-\frac{u_0'(x)}{u_0(x)}
= \frac{(\la \sqrt{x}-u_0(x))'-\hat u_0(x)u_0'(x)}{\la\sqrt{x}}\,.
\]
By the monotonicity of all the
quantities involved, we also have that
\begin{align}
  &\la\sqrt{x}-u_0(x)\leq \Big(  \sum_{j+k>0} |a^{0}_{jk}|
  \ep^{(k-1)p_0+jp_1}
  +\Big|a^0_1-\frac12\sum_{n=1}^{N_2}n^2b_n\Big|\ep^{\frac32-p_0}\\
  &\qquad\qquad\qquad+\Big|a^0_2+\frac{1}{24} \sum_{n=1}^{N_2}n^4b_n\Big|\ep^{\frac72-p_0}
  +\ep^{-\frac12-p_0} E_{\ep, u_0}\Big)x^{\frac12+p_0}
  \leq c_{\ep,p_0}x^{\frac12+p_0}\,,\nonumber\\
&  (\la \sqrt{x}-u_0(x))'-\hat u_0(x)u_0'(x)\leq
\Big(  \sum_{j+k>0} |a^{1}_{jk}|
\ep^{(k-1)p_0+jp_1}
+\Big|a^1_0-\frac12\sum_{n=1}^{N_2}n^2b_n\Big|\ep^{\frac32-p_0}\\
  &\qquad\qquad\qquad+\Big|a^1_1+\frac{1}{24} \sum_{n=1}^{N_2}n^4b_n\Big|\ep^{\frac72-p_0}
+\ep^{\frac12-p_0}E_{\ep,u_0'}\Big)x^{p_0-\frac12}
  \leq c_{\ep,p_0}' x^{p_0-\frac12}\,,\nonumber
\end{align}
where $E_{\ep,u_0}$ (resp. $E_{\ep,u_0'}$) denotes the RHS of~\eqref{erroru0}
(resp.~\eqref{erroru0'}) evaluated at $x=\ep$ and
where the numbers $c_{\ep,p_0}, c_{\ep,p_0}'$ are obtained letting
$\ep=0.1$ in the above bounds.
\end{proof}

For the arguments of the subsequent sections,
we need to show that not only $F_0$, but also its first and second order
(weighted) derivatives
\begin{equation}\label{F1F2}
  F_1(x)=F_0'(x)\,,\quad F_2(x)=|x|F_0''(x)\,,
\end{equation}
are bounded and small near $x=0$.
This is the content of the following lemma, whose proof is
omitted as it follows the same scheme of Lemma~\ref{L.F0},
that is, it relies on the asymptotic analysis of $\L u_0-u_0^2$ given
in~\eqref{Lu0-u02} and the estimates of the Lemmas~\ref{L.estu0},
\ref{L.estLu0} and \ref{L.upperbounds}. See Appendix~\ref{S.AppendixB}
for more details.

\begin{lemma}\label{L.F1F2}
  Let $u_0$ be as in~\eqref{u0}, in which
  the coefficients $a_{jk}$ and $b_{n}$ satisfy the relations
  \begin{align*}
    &a_{00}-\frac{1}{4}=0\,,\\
   & A^0_{01}-2a^0_{00}a^0_{01}=0\,,\\
 &   A^0_{1}-\frac{1}{2}\Big( \sum_{n=1}^{N_2}b_n n^{3/2}
\sqrt{\tanh(n)}-
\sum_{n=1}^{\infty} \sum_{k=0}^{N_j}\sum_{j=0}^1 a_{jk}
\frac{\sqrt{\tanh(n)}}{n^{kp_0+jp_1}}\Big)=0\,,\\
&A^0_{02}-(a^0_{02})^2-2a^0_{00}a^0_{02}=0\,,\\
&a^0_1-\frac12\sum_{n=1}^{N_2}n^2b_n=0\,.
  \end{align*}
  Then $F_1,F_2\in L^\infty(\T)$ and
\begin{equation}\label{de1,de2}
  \de_1:=\|F_1\|_{L^\infty(\T)}\leq 9.2\cdot 10^{-7}\,,\quad 
  \de_2:=\|F_2\|_{L^\infty(\T)}\leq 1.2\cdot 10^{-5}.
  \end{equation}
\end{lemma}
\section{Analysis of the linearized equation}\label{S.lineareq}
As we discussed in the introduction, one of the key elements in this
work is that we are able to exploit the (highly nontrivial)
invertibility of the linear operator that renders the reduced Whitham equation.
What is more,
the linearized equations for the derivatives of the solution are also
given by operators that we can invert and that we will study in the following section.

Indeed, it is clear that equation~\eqref{v0} suggests to invert the linear
operator $I-T_0$ to show the existence of a function $v_0\in
L^\infty(\T)$ that allows us to express a solution
of the reduced Whitham equation~\eqref{reducedW} as
\begin{equation}\label{perturbationv0}
u(x)=u_0(x)+|x|v_0(x)\,,
\end{equation}
with $u_0(x)$ our approximate solution~\eqref{u0}.
Although this ansatz by itself is not sufficient to prove the first
part of the Conjecture~\ref{conj} on the asymptotic behavior of
Whitham waves,
as we shall see in the next section one can
argue that the continuity of all the estimates
on a small parameter $\eta>0$ associated to the weight $|x|^{1+\eta}$
is sufficient to obtain the conclusion.

To begin with this analysis, in the following lemma we show that the
norm of the operator $T_0$ is smaller than~1. For notational simplicity, here and in
what follows we will denote by~$\|T\|$ the $L^\infty(\TT)\to
L^\infty(\TT)$ norm of a linear operator~$T$.

\begin{lemma}\label{L.normT0}
Let $C_B$ be the constant given by
\begin{equation}\label{cb}
C_B:=\int_0^\infty\Big|
\frac{1}{\sqrt{1-t}}+\frac{1}{\sqrt{1+t}}-2\Big|t^{-5/2}\,dt
=0.997362\ldots
\end{equation}
The number $C_B$, which can be computed explicitly as the root of a
quartic polynomial, coincides with the norm of the operator $T_0$:
  \begin{equation}
    \label{normT0}
      \|T_0\|
      =C_B\,.
\end{equation}

\end{lemma}

\begin{proof}
  Let us start with the computation of $C_B$. For convenience we split
  $C_B=c_B^1+c_B^2$ as
  \begin{align*}
c_B^1:&=\frac{1}{\pi}\int_0^1\Big(
\frac{1}{\sqrt{1-t}}+\frac{1}{\sqrt{1+t}}-2\Big)t^{-5/2}\,dt\,,\\
c_B^2:&=
\frac{1}{\pi}\int_1^\infty\Big|
\frac{1}{\sqrt{t-1}}+\frac{1}{\sqrt{1+t}}-2\Big| t^{-5/2}\,dt\,.
\end{align*}
Notice now that the first integral is immediate,
\[
   c_B^1=\int_0^1\Big(
\frac{1}{\sqrt{1-t}}+\frac{1}{\sqrt{1+t}}-2\Big)t^{-5/2}\,dt=
\frac{2}{3} (\sqrt{2}+2)\,.
\]
Furthermore,
a simple analysis of the sign of the integrand
\[
I(t):=\frac{1}{\sqrt{t-1}}+\frac{1}{\sqrt{1+t}}-2
 \] 
 reveals that $I(t)$ is positive when $1<t<t^*$,
 where $t^*$ denotes the largest (real) root of the 
 quartic polynomial $4 t^4-4 t^3-8 t^2+4 t+5$.
 Then,
 \[
c_B^2=\int_1^{t^*}I(t)\,dt-\int_{t^*}^\infty I(t)\,dt=0.857162\ldots\,,
   \]
   so that summing both contributions we see that
   $C_B$ takes the value of~\eqref{cb}.

   From the expression of the kernel of~$T_0$
   (which is even by equation~\eqref{T0}), it is
standard that the norm of~$T_0$ is
\[
\|T_0\|:=  \sup_{0<x<\pi} \frac{1}{2|x|u_0(x)}\int_{0}^{\pi} \big| K(x-y)+K(x+y)-2K(y)\big| y\, dy\,.
 \] 
Here we have used a simple parity argument to ensure we can take
$x,y>0$ and analyze
  separately the integral in~\eqref{normT0} over the regions $x<y<\pi$
  and $0<y<x$.
  We will show next that the supremum of the above expression in the
  interval $0<x\leq\ep$, where $\ep\in (0,1)$ is a certain number, is
  attained at $x=0$, where the above integral takes the
  value~$C_B$. To compute the supremum in the interval $\ep<x<\pi$ we then proceed as explained in Appendix~\ref{S.AppendixB}. 
  
   By the formula~\eqref{kernel2} of the Whitham kernel $K$,
   we notice that
\begin{equation*}
  K(x-y)+K(x+y)-2K(y)=\frac{2}{\pi} \sum_{n=1}^\infty m(n)
  (\cos(n x)-1)\cos(n y)\,.
\end{equation*}
Moreover, this expression is positive when $y>x$ by Lemma~\ref{L.signT}.
Therefore, by the definition~\eqref{Clausen} of Clausen functions,
\begin{multline}
  \label{intT00}
  \frac{1}{2x u_0}\int_{x}^{\pi} \big|
 K(x-y)+K(x+y)-2K(y)\big|y\,dy\\
 =\frac{1}{\pi xu_0} \sum_{n=1}^\infty \frac{m(n)}{n^2}
  (\cos(n x)-1) \big((-1)^n-n x \sin (n x)-\cos (n x)\big)\\
  = \frac{1}{\pi x u_0} \Big( x S_{\frac{3}{2}}(x)-\frac{x}{2} 
  S_{\frac{3}{2}}(2 x) + \frac{\sqrt{2}-2}{4}
  (C_{\frac{5}{2}}(2x)-\ze(5/2) )\\
+\sum_{n=1}^\infty \frac{1-\sqrt{\tanh(n)}}{n^{5/2}}
  (1-\cos(n x)) \big((-1)^n-n x \sin (n x)-\cos (n x)\big)
  \Big)
\end{multline}
Using now the estimates proved in \eqref{boundsClausen0},
we readily find that
\begin{equation}
\frac{1}{2 x u_0}\int_{x}^{\pi} \big|
K(x-y)+K(x+y)-2K(y)\big|y\,dy=
c_B^1-\frac{2 (1-\sqrt{2}) \ze(1/2)}{\pi^{3/2}}(1+\hat u_0(x))\sqrt{x}+E_{T_0}^1(x)\,,
\end{equation}
with $\hat u_0$ the auxiliary function~\eqref{hatu0}
and where the error term $E_{T_0}^1(x)$ can be estimated as
\begin{multline*}
|E_{T_0}^1(x)|\leq c_B^1 \hat u_0(x)+\frac{1}{4\la}\Big(
\frac{10\sqrt{2}\ze(5/2)}{4\pi^2-x^2}+\frac{3}{\pi}
\sum_{n=1}^\infty n^{3/2}(1-\sqrt{\tanh(n)})\Big) (1+\hat
u_0(x))x^{5/2}\\
=c_B^1 \hat u_0(x)+
\frac{1}{\la} c_{T_0}^1(1+\hat u_0(x))x^{5/2}\,.
\end{multline*}

For the region $0<y<x$, we rewrite the integrand in terms of
Clausen functions and then make use of the asymptotic formulas in
order to obtain an explicit error term that is small when $x< \ep$.
In fact, observe first that
\[
  \sum_{n=1}^\infty\frac{1}{\sqrt{n}} (\cos(n x)-1)\cos(n y)=
  \frac12 \big(C_{\frac12}(x-y)+C_{\frac12}(x+y)-2 C_{\frac12}(y)\big)\,.
\]
Hence, by the series representation~\eqref{ClausenC},
\begin{multline*}
  \label{intT01}
\sqrt{\frac{2}{\pi}}\frac{1}{x^{3/2}} \int_{0}^{x} \big|
  K(x-y)+K(x+y)-2K(y)\big|y\,dy\\
  =\frac{1}{\pi x^{3/2}}\int_{0}^{x} \Big|
\frac{1}{\sqrt{x-y}}+\frac{1}{\sqrt{x+y}}-\frac{2}{\sqrt{y}}
+\sqrt{\frac{2}{\pi}}
\big(E_{C_{\frac12}}(x-y)+E_{C_{\frac12}}(x+y)-2E_{C_{\frac12}}(y)\big)\\
+2 \sqrt{\frac{2}{\pi}}
\sum_{n=1}^{\infty}\frac{1-\sqrt{\tanh(n)}}{\sqrt{n}}(1-\cos(nx))
\cos(ny)  \Big|y\,dy
\end{multline*}
Using now the fact that
$ E_{C_{\frac12}}(x-y)+E_{C_{\frac12}}(x+y)-2E_{C_{\frac12}}(y)>0$ and
the formula of $E_{C_\frac12}(x)$,
we obtain that
\begin{multline*}
 \int_{0}^x \big(E_{C_{\frac12}}(x-y)+E_{C_{\frac12}}(x+y)-2E_{C_{\frac12}}(y)\big)y\,dy\\
  =\sum_{m=1}^\infty\frac{(-1)^m}{(2m)!}\ze(1/2-2m)
  \int_{0}^x \big(|x-y|^{2m}+|x+y|^{2m}-2y^{2m}\big)y\,dy\\
  = 2\sqrt{2}\sum_{m=1}^\infty
  \ze(1/2+2m)\frac{\Ga(1/2+2m)}{\Ga(1+2m)}
  \frac{m (4^m-1) }{(m+1) (2 m+1)}
  (2\pi)^{-1/2-2m}x^{2 m+2}
  \leq c_\ep x^{4}
\end{multline*}
where the constant $c_\ep$ is given by
\begin{multline}
  c_\ep=\frac{f^{(iv)}(\ep)}{4!} \,,\\
f(x)=\frac{2}{3} \sqrt{\pi } \Big(\sqrt{2} \sqrt{\pi ^2-x^2}
  \sqrt{\sqrt{\pi ^2-x^2}+\pi }-5 \sqrt{2} \pi  \sqrt{\sqrt{\pi
      ^2-x^2}+\pi }\\
  -2 \sqrt{4 \pi ^2-x^2} \sqrt{\sqrt{4 \pi ^2-x^2}+2 \pi }+20 \pi
  \sqrt{\sqrt{4 \pi ^2-x^2}+2 \pi }-24 \pi ^{3/2}
  \Big)
\ze(5/2)\,.
\end{multline}
Furthermore, since
\begin{equation*}
 \Big| \sum_{n=1}^{\infty}
  \frac{1-\sqrt{\tanh(n)}}{\sqrt{n}}(1-\cos(nx)) \cos(ny) \Big|\leq
  \frac{x^2}{2}\sum_{n=1}^{\infty} n^{3/2}\big(1-\sqrt{\tanh(n)}\big)\,,
\end{equation*}
we arrive at the estimate
\begin{equation}
\frac{1}{2|x|u_0}\int_{0}^{x} \big|
K(x-y)+K(x+y)-2K(y)\big|y\,dy\\
\leq c_B^2+E_{T_0}^{2}(x),
\end{equation}
where
\begin{multline*}
  |E_{T_0}^{2}(x)|\leq c_B^2\hat u_0(x)+
  \frac{1}{4\pi\la} \Big(2c_\ep+\sum_{n=1}^{\infty} 
  n^{3/2}\big(1-\sqrt{\tanh(n)}\big)\Big)
  (1+\hat u_0(x))|x|^{5/2}\\
  = c_B^2\hat u_0(x)+
 \frac{1}{\la} c_{T_0}^2(1+\hat u_0(x))|x|^{5/2}
\end{multline*}

In this way, to obtain that $\|T_0\|= C_B$,
we need first to verify that
\[
 E_{T_0}^1(x)+E_{T_0}^2
  -\frac{2 (1-\sqrt{2}) \ze(1/2)}{\pi^{3/2}}(1+\hat
  u_0(x))\sqrt{x}\leq 0
\]
in the range $0<x<\ep$ for sufficiently small $\ep$.
This in turn follows from the bounds that we have derived here
together with Lemma~\ref{L.upperbounds}
and the numerical inequality
\begin{equation}\label{ineqT0}
C_B c_{\ep,p_0}\ep^{p_0-1/2}+
  (c_{T_0}^1+c_{T_0}^2)\ep^2<\frac{\sqrt{2}-2}{2\pi}\ze(1/2)\,.
\end{equation}
See Appendix~\ref{S.AppendixB} for more details and also how to deal with
the case $x\geq \ep$.
\end{proof}
Using this lemma, the inverse on $ L^{\infty}(\T)$ of the
operator $I-T_0$ can be written as a Neumann series with norm
bounded as $\|(I-T_0)^{-1}\|_{L^{\infty}}\leq \be$, where
\begin{equation}
  \label{beta}
\be:=\frac{1}{1-C_B}=379.017\ldots
\end{equation}
is a parameter that we will use hereafter.
It is well known that if we show that the mapping
$G_0:L^\infty(\T)\to L^\infty(\T)$,
\begin{equation}
  \label{G0}
  v_0 \mapsto G_0(v_0):=(I-T_0)^{-1}\Big(F_0-\frac{|x|}{2u_0} v_0^2\Big)\,,
\end{equation}
is contractive and takes the ball of radius $\ep_0$ in
$L^{\infty}(\T)$ into itself,
then the existence of a solution $v_0$ of \eqref{v0}
is guaranteed by the Banach fixed point theorem.
More precisely,
letting
\[
X_{\ep_0} :=\{v_0\in L^{\infty}(\T):\, v_0(x)=v_0(-x)\,,\ 
\|v_0\|_{L^{\infty}(\T)}\leq \ep_0 \}
\]
be the functional space on which we consider \eqref{v0},
the next result holds for the constants $\be$ and $\de_0$ of before:
\begin{prop}\label{P.ep0}
Let $u_0$ be the approximate solution~\eqref{u0} of the reduced
Whitham equation~\eqref{reducedW}
for which its associated defect $\de_0=\|F_0\|_{L^{\infty}(\T)}$
is bounded as
\[
\de_0\leq \frac{1}{4\al_0\be^2}\,,\quad \al_0:=\sup_{x\in\T}\Big|\frac{x}{2u_0(x)}\Big|\,.
\]
Then, for a sufficiently small radius $\ep_0$ such that
\[
  \frac{1-\sqrt{1-4\al_0\be^2\de_0}}{2\al_0\be}\leq \ep_0\leq \frac{1}{2\al_0\be}\,,
\]
the following statements are true:

\begin{itemize}
  \item[(1)] $ G_0(X_{\ep_0})\subseteq X_{\ep_0}$.
  \item[(2)] $\|  G_0 (v_0)- G_0( w_0)\|_{L^{\infty}(\T)}\leq
    k_0 \|v_0-w_0\|_{L^{\infty}(\T)}$ with $k_0<1$
for all $v_0\,,w_0$ in $X_{\ep_0}$.
\end{itemize}
\end{prop}
\begin{proof}
  As shown in Lemma~\ref{L.alfas}, the estimate for $\hat u_0$ given
  in Lemma~\ref{L.hatu0} yields that $\al_0\leq 2.696$.
  Moreover, by Lemma~\ref{L.F0},
  \[
    \de_0\leq 9.1\cdot 10^{-8}<\frac{1}{4\al_0\be^2}
    =5.2\cdot 10^{-7}\,.
  \]
  Using now~\eqref{G0},
  it is not difficult to show that the first
  condition above is equivalent to the inequality
  $\be\big( \de_0+\al_0\ep_0^2  \big)\leq \ep_0$,
  which holds in view of the bound from below for $\ep_0$,
  and the fact that the operator $T_0$ takes even
  functions into even functions, with $u_0$ and $F_0$ also even functions
  by construction.

  Moreover,
  \[
    \|  G_0(v_0)- G_0(w_0)\|_{L^{\infty}(\T)}\leq
    \be \sup_{x\in\T}\Big| \frac{x}{2u_0(x)} \big(v_0^2-w_0^2\big)
    \Big|
    \leq 2\al_0\be\ep_0\| v_0-w_0\|_{L^{\infty}(\T)}\,,
   \]
   which by the bound from above for $\ep_0$
   makes $k_0<1$ and completes the proof.
\end{proof}


\section{Convexity}\label{S.convexity}

In this section we prove the strong statement of the conjecture on
Whitham waves, namely the convexity of the highest cusped
travelling wave solution to~\eqref{Wphi}.
To this end, we extend the argument of the previous section and show
the existence of fixed points for mappings $G_1,G_2$,
analogous to the above nonlinear map $G_0$,
which are associated to small perturbations $v_1,v_2$ of the first and
second order derivatives of the solution of the reduced Whitham
equation~\eqref{reducedW}.
The conclusion of the main Theorem~\ref{mainthm} will follow then
from the smallness of the perturbation and the convexity of our
approximate solution $u_0$.

Let us begin by considering operators
$T_i: L^{\infty}(\T)\to L^{\infty}(\T)$ with $i=1,2$
which will play the same role as $T_0$ in~\eqref{v0}:

\begin{align}
T_1 v_1(x)&=\frac{1}{2 u_0(x)}\int_{0}^{\pi} \Big(
K(x-y)-K(x+y)+\frac{u_0'(x)}{u_0(x)}K_1(x,y) \Big) v_1(y)
            \,dy  \label{T1}\\
 T_2 v_2(x)&=\frac{|x|}{2 u_0(x)}\int_{0}^{\pi} \Big( K(x-y)+K(x+y)+\frac{2 u_0'(x)}{u_0(x)} K_2(x,y) \label{T2} \\
&\qquad+\frac{1}{(u_0(x))^2}\big(2(u'_0(x))^2-u_0(x)u_0''(x) \big)\bar K_2(x,y)
-\chi(x,y) f(x) \Big) \frac{v_2(y)}{y} \,dy \nonumber\,, 
\end{align}
where we have introduced the following functions of the Whitham kernel $K$,
\begin{align}
  \label{Kbars}
K_1(x,y)&=\int_{0}^{x+y}K(t)\, dt-\int_{0}^{x-y}K(t)\, dt-2
  \int_{0}^{y}K(t)\, dt\,,\\
K_2(x,y)&=-\int_{0}^{x+y}K(t)\, dt-\int_{0}^{x-y}K(t)\, dt\,,\\
\bar K_2(x,y)&=\int_{0}^{x-y}\int_{0}^s K(t)\, dt\, ds+\int_{0}^{x+y}\int_{0}^s K(t)\, dt\, ds-2 \int_{0}^{y}\int_{0}^s K(t)\, dt\, ds\,,
\end{align}
the step function $\chi(x,y)$ that is~1 when $y<x$ and zero otherwise,
and where
\begin{equation}
  \label{fT2}
 f(x)= 2K(x)+\frac{2
  u_0'(x)}{u_0(x)} K_2(x,0) \\+\frac{1}{(u_0(x))^2}\big(2(u'_0(x))^2-u_0(x)u_0''(x) \big)\bar K_2(x,0)\,.
\end{equation}

In the next lemma we show that the nuclei of the
three operators $T_i$ have definite sign when $y>x$ when written as above. This feature will be remarkably
useful in the computation of the norms
$\|T_i\|$ as we shall see now.
\begin{lemma}\label{L.signT}
  Let $u_0$ be our approximate solution~\eqref{u0}. Then,
  \begin{align}
&K(x-y)+K(x+y)-2K(y)\,, \label{Kpos0}\\
&K(x-y)-K(x+y)+\frac{u_0'(x)}{u_0(x)}K_1(x,y)\,, \label{Kpos1}\\ 
&K(x-y)+K(x+y)+\frac{2u_0'(x)}{u_0(x)} K_2(x,y)+\frac{1}{(u_0(x))^2}\big(2(u'_0(x))^2-u_0(x)u_0''(x) 
  \big)\bar K_2(x,y) \,, \label{Kpos2}
\end{align}
are positive functions for $y>x$.
\end{lemma}
\begin{proof}
By parity considerations and the
representation formula of the Whitham Kernel~\eqref{kernel1}
that stems from~\cite[Eq. 2.18]{EW},
it is enough to check that
\begin{multline*}
\sinh\big(s(\pi-y)\big)\Big( \sinh(s x)+\frac{(1-\cosh(s
  x))}{s}\frac{u_0'(x)}{u_0(x)}\Big)\geq 0\,,\\
\cosh\big(s(\pi-y)\big) \Big( \cosh(sx)-2\frac{u_0'(x)}{u_0(x)} \frac{\sinh(sx)}{s}\\
+\frac{1}{(u_0(x))^2}\big(2(u'_0(x))^2-u_0(x)u_0''(x) \big)\frac{(\cosh(sx)-1)}{s^2}
  \Big)\geq 0
\end{multline*}
for $y>x>0$ and $s>0$.
In fact, the proof relies on the fact that the sign in the three
expressions~\eqref{Kpos0},~\eqref{Kpos1} and~\eqref{Kpos2} depends on
the sign of some combinations of the function $\cosh [s(\pi-|x|)]$
(and its derivatives) that appears in the integrand
of~\eqref{kernel1}.

Notice first that~\eqref{Kpos0} is positive as
\begin{multline*}
  \cosh\big(s(\pi -x -y)\big)+\cosh\big(s(\pi +x -y)\big)
  -2 \cosh\big(s(\pi -y)\big)\\
  =4\cosh\big(s(\pi-y)\big)
  \sinh(\tfrac12 s x)^2>0
\end{multline*}
for all $s>0$.
Furthermore,
since the functional inequality
\[
\al \sinh(z)+\frac{1-\cosh(z)}{2z}\geq 0
\]
holds for all $z>0$ and $\al>1/4$, the positivity of~\eqref{Kpos1} follows
immediately by the bound~\eqref{L.upperbound1} stated in Lemma~\ref{L.upperbounds}.
Analogously, for the last expression we use~\eqref{L.upperbound2},
the above inequality and the fact that
\[
\cosh(z)-\frac{1}{z}\sinh(z)+\frac{3}{4 z^2}\big( \cosh(z)-1 \big)\geq 0\,.
\]
\end{proof}
\begin{lemma}\label{L.normT1}
  The norm of the operator $T_1$ satisfies that
  \begin{equation}
    \label{intT1}
    \|T_1\|=C_B\,,
  \end{equation}
  where $C_B$ is the constant defined in~\eqref{cb}.
\end{lemma}
\begin{proof}
  As in the proof of the norm of $T_0$,
  we divide the integral \eqref{intT1} into two pieces and make use of the
  bounds for the Clausen functions to show that the sum is bounded by $C_B$
  for $x\leq \ep$, while for $x> \ep$ the proof is detailed in
  Appendix~\ref{S.AppendixB}.

  Let us first analyse the integral when $x<y<\pi$ (as before, we can
  assume that $x$ and $y$ are positive by parity):
  \begin{multline}
    \label{intT10}
\frac{1}{2 u_0(x)}\int_{x}^{\pi} \Big(
K(x-y)-K(x+y)+\frac{u_0'(x)}{u_0(x)}K_1(x,y) \Big)\,dy\\
=\frac{1}{\pi u_0(x)}\Big[
S_{\frac32}(x)+\frac{1-\sqrt{2}}{2} S_{\frac32}(2x)
+\frac{2-\sqrt{2}}{4}
\big( C_{\frac52}(2x)-\ze(5/2)\big)\frac{u_0'(x)}{u_0(x)}\\
+\sum_{n=1}^\infty \frac{1-\sqrt{\tanh(n)}}{n^{3/2}}\Big(
\sin(n x)+\frac{u_0'(x)}{ n u_0(x)}\big( \cos(n x) -1 \big)\Big)
\big((-1)^n- \cos(n x)\big)
\Big]\,.
\end{multline}
Noticing that
\begin{multline*}
\sum_{n=1}^\infty \frac{1-\sqrt{\tanh(n)}}{n^{3/2}}\Big(
\sin(n x)+\frac{u_0'(x)}{ n u_0(x)}\big( \cos(n x) -1 \big)\Big)
\big((-1)^n- \cos(n x)\big)\\
\leq\frac54 x \sum_{n=1}^\infty \frac{|1-(-1)^n|}{\sqrt{n}}\big(1-\sqrt{\tanh(n)}\big)
+\frac58 x^3 \sum_{n=1}^\infty n^{3/2}\big(1-\sqrt{\tanh(n)}\big)\,,
\end{multline*}
and using Lemma~\ref{boundsClausen0} combined with Lemma~\ref{L.upperbounds},
we then have that
\begin{multline}
  \frac{1}{2 u_0(x)}\int_{x}^{\pi} \Big(
  K(x-y)-K(x+y)+\frac{u_0'(x)}{u_0(x)}K_1(x,y) \Big)\,dy
  \\
  \leq c_B^1-\frac{1}{4\pi\la }\Big(
  3(\sqrt{2}-2)\ze(1/2)+
  5\sum_{n=1}^\infty\frac{|1-(-1)^n|}{\sqrt{n}}
  \big(1-\sqrt{\tanh(n)}\big)
  \Big)(1+\hat u_0(x))\sqrt{x} +E_{T_1}^1(x)\\
  =c_B^1-\frac{1}{\la}c_{\frac12} (1+\hat u_0(x))\sqrt{x}+E_{T_1}^1(x)\,,
\end{multline}
with
\begin{multline*}
  |E_{T_1}^1(x)|\leq c_B^1\hat u_0(x)
  +\frac{c_{\ep,p_0}' (\sqrt{2}-1) }{\la \pi\sqrt{\pi}} \Big(
  \frac{2\sqrt{2\pi}}{3}+\frac{|\ze(1/2)|}{2}\sqrt{x}
  +\frac{|\ze(5/2)|}{\sqrt{\pi}}\frac{x^{5/2}}{4\pi^2-x^2}
  \\
   +\frac{5 \sqrt{\pi}}{8 c_{\ep,p_0}' (\sqrt{2}-1)}\sum_{n=1}^\infty
  n^{3/2}\big(1-\sqrt{\tanh(n)}\big)x^{5/2-p_0}
  \Big)(1+\hat u_0(x))x^{p_0}\\
\leq  c_B^1\hat u_0(x)+\frac{1}{\la}c_{\ep,p_0}' c_{T_1}^1(1+\hat u_0(x))x^{p_0}
  \,,
\end{multline*}
and where we have used that
\[
  \frac{1}{\pi \sqrt{x}}\int_{x}^\pi \Big(
  \frac{1}{\sqrt{y-x}}-\frac{1}{\sqrt{x+y}}+
\frac{1}{x} \big(\sqrt{x+y}+\sqrt{y-x}-2 \sqrt{y}\big)\Big)\,dy=c^1_B\,.
\]

As for the integral in the region $0<y<x$,
the proof relies on the formula
\begin{multline*}
  \sum _{n=1}^{\infty } \frac{1}{\sqrt{n}}\sin (n y) \Big(\sin (n
 x)+\frac{1}{2 n x}\big(\cos (n x)-1\big)\Big)\\
 =\frac12
 \Big(C_{\frac12}(x-y)-C_{\frac12}(x+y)+
\frac{1}{2x}\big(S_{3/2}(x+y)-S_{3/2}(x-y)-2 S_{3/2}(y)\big)\Big)\,,
\end{multline*}
the estimates for the Clausen functions stemming from Lemma~\ref{boundsClausen0}
and the value of the integral
\[
\frac{1}{\pi \sqrt{x}}\int_{0}^x\Big|\frac{1}{\sqrt{x-y}}-\frac{1}{\sqrt{x+y}}+
\frac{1}{x} \big(\sqrt{x+y}-\sqrt{x-y}-2 \sqrt{y}\big)\Big|\,dy=c^2_B\,.
\]
In fact, we have
\begin{multline}
 \frac{1}{2 u_0(x)}\int_{0}^{x} \Big|
 K(x-y)-K(x+y)+\frac{u_0'(x)}{u_0(x)}K_1(x,y) \Big|\,dy\\
 =\frac{1+\hat u_0(x)}{\pi \sqrt{x}}\int_{0}^{x} \Big|
\frac{1}{\sqrt{x-y}}-\frac{1}{\sqrt{x+y}}+
\frac{1}{x} \big(\sqrt{x+y}-\sqrt{x-y}-2 \sqrt{y}\big)\\
+2\Big( \frac{u_0'(x)}{u_0(x)}-\frac{1}{2x}\Big)\big( \sqrt{x+y}-\sqrt{x-y}-2 \sqrt{y}\big) \\
+\sqrt{\frac{2}{\pi}}\Big(
E_{C_{1/2}}(x-y)-E_{C_{1/2}}(x+y)+\frac{u_0'(x)}{u_0(x)}\big(
E_{S_{3/2}}(x+y)-E_{S_{3/2}}(x-y)-2E_{S_{3/2}}(x+y)
\big)\Big)\\
+2 \sqrt{\frac{2}{\pi}}\sum _{n=1}^{\infty }
\frac{1-\sqrt{\tanh(n)}}{\sqrt{n}}\sin (ny)
\Big(\sin (n x)+\frac{u_0'(x)}{nu_0(x)}\big(\cos (n x)-1\big)\Big)
\Big|\,dy
\leq c_B^2+E_{T_1}^2(x)\,,
\end{multline}
where
\begin{multline}
  |E_{T_1}^2(x)|\leq c_B^2\hat u_0+
  \frac{8 c_{\ep,p_0}'}{15\pi\la}(5-5\sqrt{2}+2\sqrt{5})(1+\hat u_0(x))x^{p_0}\\
  +\frac{1}{\la}\Big(\frac{5}{4\pi}\sum _{n=1}^{\infty }
  n^{1/2}\big(1-\sqrt{\tanh(n)}\big)+c_{\ep}'x\Big)(1+\hat u_0(x))x^{3/2}\\
\leq  c_B^2\hat u_0+\frac{1}{\la}c_{\ep,p_0}' c_{T_1}^2(1+\hat u_0(x)) x^{p_0}+\frac{1}{\la}c_{T_1}^3(1+\hat u_0(x)) x^{3/2}
  \,,
\end{multline}
where the number $c_\ep'$ is a bound for the integrals coming from the
Clausen error terms:
\begin{multline*}
  c_\ep':=\frac{g'''(\ep)}{3!}\,,\\
  g(x):=\frac{x}{\sqrt{\pi }}  \Big(\frac{\sqrt{2}}{\sqrt{\sqrt{\pi
          ^2-x^2}+\pi }}-\frac{2}{\sqrt{\sqrt{4 \pi ^2-x^2}+2 \pi
      }}\Big) \ze(5/2)\\
  +\frac{2}{3 \sqrt{\pi }x}
    \Big(x \big(-\sqrt{\pi -x}+\sqrt{x+\pi }-\sqrt{2} \sqrt{x+2 \pi}
        +\sqrt{4 \pi -2 x}\big)\\
      +\pi  \big(\sqrt{\pi -x}+\sqrt{x+\pi }-2 \sqrt{2} \sqrt{x+2 \pi }-2 \sqrt{4 \pi -2 x}\big)+6 \pi ^{3/2}\Big) \ze(5/2)\,.
\end{multline*}
Then, since the numerical inequality
\begin{equation}\label{ineqT1}
 ( C_B c_{\ep,p_0}+c_{\ep,p_0}' (c_{T_1}^1+c_{T_1}^2) )\ep^{p_0-1/2}+c_{T_1}^3\ep<c_{\frac12}\,,
\end{equation}
holds for small enough $\ep$, it follows that
$\|T_1\|= C_B$

\end{proof}



Likewise, using Lemma~\ref{L.signT} we can prove that the norm of the
operator $T_2$ on $L^\infty(\T)$ is also less than 1.
Since the proof entails some careful computations related to the
singularity $y^{-1}$ in the integrand of~\eqref{fT2},
we leave the details for Appendix~\ref{S.AppendixA}.


\begin{lemma}\label{L.normT2}
   The norm of $T_2$ also agrees with the constant $C_B$, i.e.
$\|T_2\|=C_B$.
 \end{lemma}

In what follows, we consider an exact solution $u$ of
\eqref{reducedW} and exploit the fact that we already have an
approximation $u_0$ that allows us to show the existence of a solution
with (almost) the right asymptotic behavior.
Along with the perturbation $v_0$ given in~\eqref{v0},
here we introduce bounded perturbations
$v_1$ and $v_2$ of the first and second derivatives of $u$:
\begin{equation}
  \label{perturbations}
u=u_0(x)+|x| v_0(x)\,,\quad
u'(x)=u_0'(x)+v_1(x)\,,\quad
u''(x)=u_0''(x)+\frac{v_2(x)}{|x|}\,.
\end{equation}

It is not difficult then to show that for a fixed $v_0$,
the perturbations $v_1$ and $v_2$ obey the following linear equations:
\begin{lemma}
Let $v_1\,,v_2$ be defined as above and let $F_0(x),F_1(x)$ and
$F_2(x)$ be the error terms defined in Lemmas~\ref{L.F0} and
~\ref{L.F1F2}:
\[
F_0(x)=\frac{1}{2|x| u_0}(\L u_0-u_0^2)\,,\quad F_1(x)=F_0'(x)\,,\quad F_2(x)=|x|F_0''(x)\,.
\]
Then, the functions $v_1$ and $v_2$ satisfy the system
\begin{align}
(I-T_1) v_1(x)&=F_1(x)+\frac{x^2u_0'(x)}{2u_0^2(x)}v_0^2(x)-\frac{|x|}{u_0(x)}
                v_0(x) v_1(x)\,,   \label{v1}
                \\
  (I-T_2)v_2(x)&=F_2(x)+\frac{|x|}{2 u_0(x)} f(x)v_1(x)
                 +\frac{|x|^3}{2u_0^3(x)}\big(u_0(x)u_0''(x)-2(u'_0(x))^2\big)v_0^2(x)                 \label{v2}
  \\
              &\qquad \qquad\qquad
                -\frac{|x|}{u_0(x)}v_1(x)^2+\frac{2x^2u_0'(x)}{u_0^2(x)}v_0(x)v_1(x)
                -\frac{|x|}{u_0(x)} v_0(x) v_2(x)\,.\nonumber
\end{align}
\end{lemma}
\begin{proof}
  As the proofs follow the same structure, we only give the details
  for~\eqref{v1}. Let us write $u(x)=u_0(x)+\bar u(x)$.
  Since
  \[
    \bar u(x)-\frac{1}{2 u_0(x)}\L \bar u(x)=
    F_0(x)-\frac{1}{2 u_0(x)}\bar u^2(x)\,,
  \]
  it is clear that taking $\bar u(x)=|x| v_0(x)$ we
  obtain~\eqref{v0}.
  On the other hand, differentiating in the above
  equation we have that
  \begin{multline*}
    F_1(x)+\frac{u_0'(x)}{2 u_0^2(x)}\bar u^2(x)-\frac{1}{u_0(x)}
    \bar u'(x) \bar u(x)\\
    =\bar u'(x)+ \frac{u_0'(x)}{2 u_0^2(x)}
    \int_0^\pi\big( K(x+y)+K(x-y)-2K(y)\big)\bar u(y)\,dy\\
    -\frac{1}{2 u_0(x)}\int_0^\pi \big(K'(x+y)+K'(x-y)\big)\bar
    u(y)\,dy\\
    =\bar u'(x) +\frac{u_0'(x)}{2 u_0^2(x)}\int_0^\pi\dy K_1(x,y)\bar u(y)\,dy      +\frac{1}{2 u_0(x)}\int_0^\pi \dy\big(K(x-y)-K(x+y)\big)\bar
    u(y)\,dy\,.
  \end{multline*}
  Integrating by parts again (where the boundary terms are zero by
  parity), the definition of the operator $T_1$ and letting
  $v_1(x):=\bar u'(x)$ above give~\eqref{v1}.
 The formula~\eqref{v2} simply follows by differentiating
 twice, integrating by parts and defining $v_2(x)=\bar u''(x)/|x|$.
\end{proof}

In a complete analogous manner to the previous section,
we define mappings $G_i\,, i=1,2$ on $L^{\infty}(\T)$,
\begin{align}
  \label{G1}
G_1
  v_1(x)&:=(I-T_1)^{-1}\Big(F_1(x)+\frac{x^2u_0'(x)}{2u_0^2(x)}v_0^2(x)
          -\frac{|x|}{u_0(x)}v_0(x) v_1(x)\Big)\\
G_2v_2(x)&:=(I-T_2)^{-1}\Big(F_2(x)+\frac{|x|}{2 u_0(x)} f(x)v_1(x)
                 +\frac{|x|^3}{2u_0^3(x)}\big(u_0(x)u_0''(x)-2(u'_0(x))^2\big)v_0^2(x) \nonumber
  \\
              &\qquad \qquad\qquad
                -\frac{|x|}{u_0(x)}v_1(x)^2+\frac{2x^2u_0'(x)}{u_0^2(x)}v_0(x)v_1(x)
                -\frac{|x|}{u_0(x)} v_0(x) v_2(x)\Big)\,.     \label{G2}
\end{align}
Moreover, let us write
\begin{equation}
  \label{ballepi}
X_{\ep_i} :=\{v_i\in L^{\infty}(\T):\, v_i(x)=(-1)^iv_i(-x)\,,\ 
\|v_i\|_{L^{\infty}(\T)}\leq\ep_i\}
\end{equation}
and let $\de_i=\|F_i\|_{L^{\infty}(\T)}$ for $i=0,1,2$ be the constants
of Lemmas~\ref{L.F0} and~\ref{L.F1F2}. 
Notice that by Lemmas~\ref{L.normT0},~\ref{L.normT1}
and~\ref{L.normT2}, the three norms
$\|T_i\|=C_B< 1$ so that we can find
conditions on the numbers $\ep_i$ that allow us to apply the
Banach fixed point theorem for the operators $ G_i$
and then show the existence of a solution
$u$ to~\eqref{reducedW} of the form~\eqref{perturbations}.

For this, we will need explicit control of the following numerical constants:
\begin{lemma}\label{L.alfas}
  Let
\begin{align*}
&\al_0:=\sup_{x\in \T}\Big|\frac{x}{2u_0(x)}\Big|\,,\quad
                 \al_1:=\sup_{x\in\T}\Big|\frac{x^2u_0'(x)}{2u_0^2(x)}\Big|
\,,\quad   \al_f:=\sup_{x\in\T}\Big| \frac{x}{2u_0(x)} f(x) \Big|\,,\\
&\al_2:=\sup_{x\in\T}\Big|\frac{x^3}{2(u_0(x))^2}\Big(u_0''(x)
-\frac{2(u_0'(x))^2}{u_0(x)}\Big)\Big|\,,\quad
\bar
                                                                        \al_2:=\sup_{x\in\T}\Big|\frac{2x^2u_0'(x)}{(u_0(x))^2}\Big|\,.
\end{align*}
Then, the values of these constants are
\[
  \al_0\leq 2.696\,,\al_1\leq 0.32\,,\al_2\leq 1.382\,, \bar\al_2\leq 1.280
  \,,\al_f\leq 0.448\,.
 \]
\end{lemma}
\begin{proof}
  As for the rest of quantities that we estimate through the paper,
  the bounds near $x=0$ follow from the asymptotic analysis carried
  out in Section~\ref{S.approx.sol.} and particularly from
  Lemma~\ref{L.upperbounds}. Away from zero the proof is detailed in
  Appendix~\ref{S.AppendixB}.
  
\end{proof}

Using in addition that the operator $T_1$ (rep. $T_2$) maps odd
(resp. even) functions into odd (resp. even) functions (as it is clear
from~\eqref{T1} and~\eqref{T2}),
we have the next result analogous
to Proposition~\ref{P.ep0}:
\begin{prop}\label{P.ep1ep2}
Let $u_0$ be an approximate solution of
\eqref{reducedW}
and let $\al_0,\al_1,\al_2,\bar \al_2,\al_f$ be as before.
Assume also that the original perturbation $v_0\in X_{\ep_0}$ for a
small $\ep_0$ as in Proposition~\ref{P.ep0}
and take small constants $\ep_1\,,\ep_2$ such that
\begin{align*}
\frac{
 \be (\de_1+\al_1\ep_0^2)}{1-2\al_0\be\ep_0}&\leq \ep_1\leq 3.77\cdot 10^{-4} \,,\\
\frac{
  \be (\de_2+\al_f\ep_1 +\al_2\ep_0^2+2\al_0\ep_1^2+
  \bar\al_2\ep_0\ep_1)}{1-2\al_0\be\ep_0}&\leq \ep_2\leq 7.46\cdot 10^{-2}\,.
\end{align*}
Then, for $i=1,2$ it holds that:
\begin{itemize}
  \item[(1)] $ G_i(X_{\ep_i})\subseteq X_{\ep_i}$
  \item[(2)] $\|  G_i v_i- G_i w_i\|_{L^{\infty}(\T)}\leq
    k_i \|v_i-w_i\|_{L^{\infty}(\T)}$ with $k_i<1$
for all $v_i\,,w_i$ in $X_{\ep_i}$.
\end{itemize}
\end{prop}

Now we are ready to prove the main result of the paper.
In the following, we use the fact that there exists a negative constant
$c$ such that $u_0''(x)<c/|x|<0$ everywhere.
Since $\ep_2$ is sufficiently small, the second derivative of the solution
$u$ of \eqref{reducedW} is then also signed. More precisely:
\begin{lemma}\label{L.convex}
  Let $\ep_2$ be the smallest
  radius of a ball in $L^\infty(\T)$ for which there
  exists a function $v_2$ solution to~\eqref{v2}.
  Then,
  \begin{equation}\label{u''<0}
u''(x)\leq u_0''(x)+\frac{\ep_2}{|x|}<0
\end{equation}
for $x\in [-\pi,\pi]$.
\end{lemma}
\begin{proof}
  Sufficiently close to $x=0$ the proof follows by the asymptotic
  formula~\eqref{u0''} and the bound for $\ep_2$ in
  Proposition~\ref{P.ep1ep2}.
  For $x$ bounded away from zero the proof is
  done as explained in Appendix~\ref{S.AppendixB}.
\end{proof}

Finally, we prove the convexity of the highest cusped Whitham wave:
\begin{proof}[\bf{Proof of Theorem \ref{mainthm}:}]
  From Propositions~\ref{P.ep0} and~\ref{P.ep1ep2}
  and the Banach Fixed point theorem,
  it follows that there exists an unique triplet of functions
  $v:=(v_0,v_1,v_2)\in X_{\ep_0} \times X_{\ep_1} \times X_{\ep_2}$
  related to the solution $u$ of the Whitham equation through
  \eqref{perturbations}.
  Moreover, let us take $\eta>0$
  and consider now the system
\begin{align*}
  \begin{cases}
u_\eta(x)&=u_0(x)+|x|^{1+\eta}\, v_{\eta,0}(x)\,,\\
u_\eta'(x)&=u_0'(x)+|x|^{\eta}\,v_{\eta,1}(x)\,,\\
u_\eta''(x)&=u_0''(x)+|x|^{\eta-1}\,v_{\eta,2}(x)\,,
\end{cases}
\end{align*}
where $u_\eta$ denotes a solution of the reduced Whitham
equation~\eqref{reducedW} and
$$v_\eta:=(v_{\eta,0}, v_{\eta,1}, v_{\eta,2})
\in X_{\ep_{\eta,0}} \times X_{\ep_{\eta,1}} \times X_{\ep_{\eta,2}}$$
is a small bounded perturbation with $X_{\ep_{\eta,i}}$
as in~\eqref{ballepi}.

Using dominated convergence we can argue that the corresponding
estimates for the new maps analogous to $T_i$ and $G_i$
are continuous in the parameter $\eta$,
so that the fixed point argument remains true with slightly
different constants when we take $\eta>0$ arbitrarily small.
In particular,
the estimate~\eqref{u''<0} implies that
  \begin{equation}
u_\eta''(x)<0\,,\quad \eta> 0\,.
  \end{equation}
In this manner, in order to conclude we need to check that $u_\eta(x)$
is the unique solution to Whitham's equation~\eqref{reducedW} of the
form $u_0(x)+|x|^{1+\eta}\, v_{\eta,0}(x)$, but this is clear in view
of the uniqueness of the fixed point $v_\eta$.
\end{proof}

\appendix
\section{The norm of $T_2$}\label{S.AppendixA}
This Appendix is devoted to the computation of the norm of the
operator $T_2$ given in~\eqref{T2}.
Like in the cases of $T_0$ and $T_1$,
using the asymptotic analysis carried out in Section~\ref{S.approx.sol.}
we show that $\|T_2\|$ is precisely the constant $C_B$.

 \begin{proof}[Proof of Lemma~\ref{L.normT2}]
  Arguing as in Lemma~\ref{L.normT1},
  let us take $x,y>0$ and let $\ep$ be  a small positive number.
  In this way, notice that the integrand of \eqref{T2}
  can be expressed as
   \begin{multline*}
     \frac{2}{\pi}\sum_{n=1}^\infty m(n) \Big( \cos(n x)
     -2\frac{u_0'(x)}{n u_0(x)}\sin(n x)\\
     +\frac{1}{n^2(u_0(x))^2}\big(2(u'_0(x))^2-u_0(x)u_0''(x)\big)
     \big(1-\cos(nx)\big)
     \Big) \frac{\big(\cos(n y)-\chi(x,y)\big)}{y}\\
     =\frac{1}{\pi y}\Big( C_{\frac12}(x-y) + C_{\frac12}(x+y)-2 \chi(x,y) C_{\frac12}(x)
     -2\frac{u_0'(x)}{u_0(x)}\big(S_{\frac32}(x-y)+S_{\frac32}(x+y)-2 \chi(x,y) S_{\frac32}(x)\big)\\
     -\frac{1}{(u_0(x))^2}\big(2(u'_0(x))^2-u_0(x)u_0''(x)\big)
     \big(C_{\frac52}(x-y)+C_{\frac52}(x+y)-2 C_{\frac52}(y)-2 \chi(x,y) (C_{\frac52}(x)-\ze(5/2))\big)
     \Big)\\
     +\frac{2}{\pi}\sum_{n=1}^\infty \frac{\big(\sqrt{\tanh(n)}-1\big)}{\sqrt{n}}\Big( \cos(n x)
     -2\frac{u_0'(x)}{n u_0(x)}\sin(n x)\\
     +\frac{1}{n^2(u_0(x))^2}\big(2(u'_0(x))^2-u_0(x)u_0''(x)\big)
     \big(1-\cos(nx)\big)
     \Big) \frac{\big(\cos(n y)-\chi(x,y)\big)}{y}\,.
   \end{multline*}
   Then, the norm of $T_2$ is obtained by taking the supremum in $x\in [0,\pi]$ of the integral
   \begin{multline*}
     \frac{\sqrt{x}}{\pi}(1+\hat u_0(x))\int_{0}^{\pi}
     \Big| \frac{1}{ \sqrt{|x-y|}}+\frac{1}{ \sqrt{x+y}}
     -\frac{2}{x}\big(\sqrt{x+y}+\sign(x-y)\sqrt{|x-y|}\big)\\
     +\frac{1}{x^2}\big((x+y)^{3/2}+|x-y|^{3/2}-2y^{3/2}\big)\\
     +4\Big( \frac{1}{2x}-\frac{u_0'(x)}{u_0(x)}\Big)\big(
     \sqrt{x+y}+\sign(x-y)\sqrt{|x-y|} -2 \chi(x,y)\sqrt{x}\big)\\
    -\frac43\Big(
    \frac{3}{4x^2}-\frac{1}{(u_0(x))^2}\big(2(u'_0(x))^2-u_0(x)u_0''(x)\big)
    \Big) \big((x+y)^{3/2}+|x-y|^{3/2}-2y^{3/2}-2 \chi(x,y) x^{3/2}\big)\\
    +\sqrt{\frac{2}{\pi}}\Big( E_{C_\frac12}(x-y)+E_{C_\frac12}(x+y)
    -2 \chi(x,y) E_{C_\frac12}(x)\\
    -\frac{2u_0'(x)}{u_0(x)}\big( E_{S_\frac32}(x-y)+E_{S_\frac32}(x+y)
    -2 \chi(x,y) E_{S_\frac32}(x) \big)\\
    -\frac{1}{(u_0(x))^2}\big(2(u'_0(x))^2-u_0(x)u_0''(x)\big)
    \big( E_{C_\frac52}(x-y)+E_{C_\frac52}(x+y)
    -2 E_{C_\frac52}(y)-2 \chi(x,y) E_{C_\frac52}(x)\big)\\
     +2\sum_{m=1}^\infty
   \frac{1}{\sqrt{n}}\big(1-\sqrt{\tanh(n)}\big) \Big( \cos(n x)
     -2\frac{u_0'(x)}{n u_0(x)}\sin(n x)\\
     +\frac{1}{n^2(u_0(x))^2}\big(2(u'_0(x))^2-u_0(x)u_0''(x)\big)
     \big(1-\cos(nx)\big)
     \Big)\big(\chi(x,y)-\cos(n y)\big)
    \Big)\Big|\,\frac{dy}{y}
    \,.
  \end{multline*}
As before, we carry out the analysis of the norm by dividing the above
integral into two pieces: the regions $x<y<\pi$, in which
$\chi(x,y)=0$ so the integrand is
positive by Lemma~\ref{L.signT}, and $0<y<x$.

Notice that the integral on $x<y<\pi$ results
\begin{multline*}
  \frac{x}{\pi u_0(x)}\sum_{n=1}^\infty m(n) \Big( \cos(n x)
     -2\frac{u_0'(x)}{n u_0(x)}\sin(n x)\\
     +\frac{1}{n^2(u_0(x))^2}\big(2(u'_0(x))^2-u_0(x)u_0''(x)\big) \big(1-\cos(nx)\big)\Big)
     \big( \text{Ci}(n\pi)-\text{Ci}(n x)\big)\,,
   \end{multline*}
where  $\operatorname{Ci}(n \pi)$ denotes
the cosine integral function
(see e.g. \cite [Chapter 6]{DLMF} for the definition
of $\text{Ci}(x)$ and its properties).
In this region,
we will use the integral estimates
\begin{multline*}
4\Big( \frac{1}{2x}-\frac{u_0'(x)}{u_0(x)}\Big)\int_{x}^{\pi}\big(\sqrt{x+y}-\sqrt{y-x}\big)\,\frac{dy}{y}\\
-\frac43\Big(
\frac{3}{4x^2}-\frac{1}{(u_0(x))^2}\big(2(u'_0(x))^2-u_0(x)u_0''(x)\big)\Big)
\int_{x}^{\pi}\big((x+y)^{3/2}+(y-x)^{3/2}\big)\,\frac{dy}{y}\\
\leq
 \frac{\pi \hat c_{p_0}'}{\la} x^{p_0-1/2}\,,
\end{multline*}
and
\begin{multline*}
-\frac 2 \pi \sum_{n=1}^\infty
    \frac{1}{\sqrt{n}}\big(1-\sqrt{\tanh(n)}\big) \Big( \cos(n x)
     -2\frac{u_0'(x)}{n u_0(x)}\sin(n x)\\
     +\frac{1}{n^2(u_0(x))^2}\big(2(u'_0(x))^2-u_0(x)u_0''(x)\big)
     \big(1-\cos(nx)\big)
     \Big)\int_x^\pi \frac{\cos(n y)}{y}\, dy\\
     \leq \frac{3}{4\pi} \sum_{n=1}^\infty\frac{1-\sqrt{\tanh(n)}}{ \sqrt{n}}
     \big(\log (x)+\log (n) +\ga-\text{Ci}(n \pi )\big)
     +2\pi (c^1_{p_0}-\log(x) c^2_{p_0})x^{p_0}\\
     \leq
     -\frac{1}{12}\log(\pi)\ze(1/2)+
     \frac{3}{4\pi}
     \sum_{n=1}^\infty\frac{1-\sqrt{\tanh(n)}}{\sqrt{n}}\log(x)
     +2\pi(c^1_{p_0}-\log(x) c^2_{p_0})x^{p_0}\,,
\end{multline*}
which hold for small constants $\hat c_{p_0}'\,,c_{p_0}^1\,,c_{p_0}^2$
that only depend on the bounds obtained in Lemma~\ref{L.upperbounds},
and where in the last estimate we have used that
\[
\int_x^\pi \frac{\cos(n y)}{y}\, dy=\text{Ci} (n \pi )-\text{Ci}(n
x)\geq \text{Ci}(n \pi ) -\ga-\log (n) -\log (x)\,,
\]
and also the numerical inequality
\begin{equation}\label{ineqcosint}
\sum _{n=1}^{\infty} \frac{\sqrt{\tanh (n)}-1}{\sqrt{n}} (\text{Ci}(n \pi )-\log (n)-\gamma )+\frac{ \log (\pi )}{9 \pi}  \ze(1/2)<0\,,
 \end{equation}
 (with $\ga$ the Euler constant) that we check with the computer.
   
Since in addition
   \begin{multline*}
     \int_x^\pi
     \Big( E_{C_\frac12}(x-y)+E_{C_\frac12}(x+y)
    -\frac{2u_0'(x)}{u_0(x)}\big( E_{S_\frac32}(x-y)+E_{S_\frac32}(x+y) \big)\\
    -\frac{1}{(u_0(x))^2}\big(2(u'_0(x))^2-u_0(x)u_0''(x)\big)
    \big( E_{C_\frac52}(x-y)+E_{C_\frac52}(x+y)
    -2 E_{C_\frac52}(y)\big)\,\frac{dy}{y}\\
    \leq
    \frac{3}{4} \zeta (1/2)\log \Big(\frac{\pi }{x}\Big)+2\pi c_{T_2}^1x^2\,,
   \end{multline*}
putting together the above estimates we arrive at
\begin{multline}
\frac{x}{2 u_0(x)}\int_{x}^{\pi} \Big(
K(x-y)+K(x+y)+\frac{2
  u_0'(x)}{u_0(x)} K_2(x,y)\\
+\frac{1}{(u_0(x))^2}\big(2(u'_0(x))^2-u_0(x)u_0''(x)
\big)\bar K_2(x,y) \Big)\, \frac{dy}{y}\\
\leq
c_B^1-\frac{ 1}{\la}c''_{\frac12}\sqrt{x}(1+\hat u_0(x))+E_{T_2}^1(x)
\,,
\end{multline}
where
\begin{multline}
  c''_{\frac12}=-\frac{1}{3\pi}\log(\pi)\ze(1/2)\,,\\
  |E_{T_2}^1(x)|\leq c_B^1\hat u_0(x)+
 \frac{3}{8\pi^2\la }\log(x)\Big(
 \sum_{n=1}^\infty\frac{1-\sqrt{\tanh(n)}}{\sqrt{n}}-\pi \ze(1/2)
 \Big) \sqrt{x}(1+\hat u_0(x))\\
 + \frac{1}{\la} \hat c_{p_0}'x^{p_0}(1+\hat u_0(x))
 +\frac{1}{\la}  (c^1_{p_0}-\log(x) c^2_{p_0})x^{p_0+1/2}(1+\hat u_0(x))
    +\frac{c^1_{T_2}}{\la} x^{5/2}(1+\hat u_0(x))\\
=c_B^1\hat u_0(x)+\frac{1}{\la}\tilde c_{1/2}\log(x)\sqrt{x}(1+\hat u_0(x))
\\+\frac{1}{\la}\Big( \hat c_{p_0}'+(c^1_{p_0}-\log(x)
c^2_{p_0})\sqrt{x}\Big) x^{p_0}(1+\hat u_0(x))
    +\frac{c^1_{T_2}}{\la} x^{5/2}(1+\hat u_0(x))
\end{multline}

On the other hand, by analogous estimates we find that
\begin{multline}
\frac{x}{2 u_0(x)}\int_{0}^{x} \Big|
K(x-y)+K(x+y)+\frac{2
  u_0'(x)}{u_0(x)} K_2(x,y)\\
+\frac{1}{(u_0(x))^2}\big(2(u'_0(x))^2-u_0(x)u_0''(x)
\big)\bar K_2(x,y) \Big|\, \frac{dy}{y}\\
\leq c_B^2 +E_{T_2}^2(x)\,,
\end{multline}
with
\begin{equation}
 |E_{T_2}^2(x)|\leq c_B^2\hat u_0(x)+\frac{1}{\la}\hat c_{p_0}''x^{p_0}(1+\hat u_0(x))+\frac{c_{T_2}^2}{\la}x^{5/2}(1+\hat u_0(x))\,.
\end{equation}
Then, since for sufficiently small $\ep$
\begin{equation}\label{ineqT2}
  \Big(  C_B c_{p_0} +\hat c_{p_0}'+\hat c_{p_0}'' + (c^1_{p_0}-\log(\ep)
c^2_{p_0})\sqrt{\ep}\Big) \ep^{p_0-1/2}+(c_{T_2}^1+c_{T_2}^2)\ep^{2-p_0}
<c_{\frac12}''-\tilde c_{\frac12}\log(\ep)\,,
\end{equation}
the proof follows in the same manner as in
Lemmas~\ref{L.normT0} and~\ref{L.normT1}.
 \end{proof}

\section{Technical details concerning the computer assisted part}\label{S.AppendixB}
In this section we discuss the technical details about the
implementation of the different numerical computations such as the
integrals that appear in the proofs along the paper.
We remark that we are computing explicit (but complicated) functions
over a one dimensional domain.
In order to perform the rigorous computations we used the Arb library
\cite{Johansson:Arb}; the code can be found in the supplementary material.

The implementation is split into several files, and many of the headers of the functions (such as the integration methods) contain pointers to functions (the integrands) so that they can be reused for an arbitrary number of integrals with minimal changes and easy and safe debugging.
We first describe the data structures that will appear in the different parts of the code and later get to the specific algorithms of each Lemma.

There is a basic class that encloses all the necessary information
used throughout the computations in
Lemmas~\ref{L.F0},~\ref{L.F1F2},~\ref{L.normT0},~\ref{L.normT1} and~\ref{L.normT2}.
It is called
\texttt{Integration\_params\_struct} and has the following members:
three integers, \texttt{N\_0}, \texttt{N\_1} and \texttt{N\_2}; three
vectors of intervals \texttt{a0k}, \texttt{a1k} and \texttt{bi} of
sizes $N_0, N_1$ and $N_2$ respectively, containing the coefficients
that describe the approximate solution $u_0$ of~\eqref{u0}.
There is also an interval called \texttt{x}, which is used only in
Lemmas~\ref{L.normT0},~\ref{L.normT1} and~\ref{L.normT2},
indicating the value of $x$ used for the integration.

We had to implement the Clausen functions, since they are not part of the Arb library. A naive implementation of $C_{z}(x)$ (resp. $S_{z}(x)$) would be to evaluate the real (resp. imaginary) parts of $\Li_{z}(e^{ix})$. When $x$ is an interval, this gives a disastrous error. Instead, we will make use of the following Lemma:

\begin{lemma}
  Let $z$ be a non-integer fixed real number. Then, the Clausen
  function $C_z(x)$ is strictly monotonic in $x\in(0,\pi]$.
\end{lemma}
\begin{proof}
  Notice that $C'_z(x)=-S_{z-1}(x)$ for all $x$ and
  assume first that $z>1$.
  By \cite[Eq.~25.12.11]{DLMF}, we have that
\[
S_{z-1}(x)=\frac{\sin(x) }{\Ga(z-1)} \int_{0}^\infty t^{z-1} \frac{e^t}{\big(e^t-\cos(x)\big)^2+\sin(x)^2}\,dt\,.
\]
Since $\Ga(z-1)>0$ for $z>1$ and $\sin(x)>0$ in $[0,\pi]$,
$C'_z(x)<0$ in that range.

Likewise, when $z<1$ we can use the representation formula
\begin{equation*}
S_{z-1}(x)=\sin(\tfrac{\pi}{2} z)\int _{0}^{\infty }t^{1-z}\frac{\sinh\big(t(\pi-x
  )\big)}{\sinh(\pi t)}\, dt
\end{equation*}
that follows from the well known
relationship between zeta functions and
polylogarithms, cf.~\cite[Eq.~25.11.25]{DLMF}.
\end{proof}

This shows that if $X = [\underline{x},\overline{x}] \subset (0,\pi]$, then $C_{z}(X)$ is contained in the convex hull of $C_{z}(\underline{x})$ and $C_{z}(\overline{x})$. That is exactly how we implement it. We compute $C_{z}$ at the endpoints using the polylogarithm function and we take their convex hull. In order to implement $S_{z}$ (which is not monotonic in $(0,\pi]$) we use that $S_{z}'(x) = C_{z-1}(x)$ and $S_{z}(X) = S_{z}(x_0) + (X-x_0)C_{z-1}(X)$ by virtue of the mean value theorem, choosing $x_0$ as the midpoint of $X$.

It is also important to remark that given the delicate set of calculations that need to be performed, working with double precision is not enough and multiprecision is needed. In all our calculations we worked with 100 bits (as opposed to the usual 53).

\begin{proof}[Proof of Lemma~\ref{L.p0p1}]
We will enclose a solution to~\eqref{exponents} by applying a Newton
method to the difference of the LHS and the RHS of the equation.
We discuss the details of the algorithm below.

The first step of the algorithm is to isolate the roots: this is done
by checking the signs of the endpoints and ensuring that the
derivative of the function has a definite sign between the
endpoints. On the contrary, if the signs of the function at the
endpoints are the same and the function is monotone, there is no root
in that interval and it is discarded. Finally, if none of these two
conditions are met, the interval is split by the midpoint in two and
the isolating function is called recursively with the two resulting
subintervals. The second step is to refine the interval even more
using a bisection method. Finally, a Newton zero-finding method is
applied. The code can be found in the file
\texttt{Lemma\_p0\_p1.c}. The total execution time was a few
seconds.
The initial intervals for $p_0$ and $p_1$ were $[0.5125,0.75]$
and $[2.625,2.875]$, and the final enclosures were
$0.61120158988884395 \pm 7.01\cdot10^{-19}$ and $2.7624011603378232
\pm 2.00\cdot 10^{-17}$,
respectively.

\end{proof}

\begin{proof}[Proof of Lemmas~\ref{L.hatu0},
  ~\ref{L.upperbounds}]
This concerns the proof of the inequalities~\ref{ineqT0}
and~\ref{ineqT1}, and all the Lemmas such as~\ref{L.hatu0}
which involve evaluations at a single point. We refer the
reader to the file \texttt{Constant\_checking.c}.
\end{proof}

\begin{proof}[Proof of Lemmas~\ref{L.F0},~\ref{L.F1F2}~\ref{L.alfas}
  and~\ref{L.convex}]
This describes the bounding of the quantities $\al_0$, $\al_1$, $\al_2$,
$\bar \al_2$, $\al_f$ and $\de_0$, $\de_1$, $\de_2$,
which are all done the same way.

We start splitting $I = [0,\pi]$ into two pieces, $I_1 =
[0,\varepsilon]$ and $I_2 = [\varepsilon,\pi]$, with $\varepsilon =
10^{-2}$. The bounds of the different quantities over $x \in I_1$ were
obtained using asymptotics for small $x$ (see e.g. Lemmas~\ref{L.estu0}
and~\ref{L.estLu0}). In order to deal with the case $x \in I_2$, we constructed a function called \texttt{compute\_bound\_Linfty\_norm\_C1} that takes as arguments a function \texttt{func}, its derivative \texttt{deriv}, a bound \texttt{bound}, an interval \texttt{min\_width} and an interval \texttt{inp} and performs recursively the following branch and bound algorithm: we first compute an enclosure of \texttt{func} (which we call \texttt{F}). The enclosure is a $C^1$ one, given by

\begin{align*}
F(X) = F(x_0) + (X-x_0)F'(X),
\end{align*}

taking $x_0$ as the midpoint of $X$. Given \texttt{F}, the function performs the following algorithm:

\begin{itemize}
\item If $\texttt{F} > \texttt{bound}$ it returns false
\item If $\texttt{F} < \texttt{bound}$ it returns true
\item If none of the two conditions are met:
\begin{itemize}
  \item If $width(\texttt{inp}) < \texttt{min\_width}$, split into two pieces and return true if both true, otherwise false
  \item Else return false
\end{itemize}

\end{itemize}

It is clear that if the algorithm returns true, then \texttt{bound} is
a guaranteed upper bound of $f(x), x \in I_2$.
For all the above quantities, the total time of computation was a few minutes. The code can be found in the file \texttt{Lemma\_bound\_functions.c}.
\end{proof}

\begin{proof}[Proof of Lemmas~\ref{L.normT0},~\ref{L.normT1} and ~\ref{L.normT2}]

We now explain how the integrals are calculated. For simplicity, we
will explain how to calculate $T_0$ but the same method applies to
$T_1$ and $T_2$. First, we split the interval $I = [0,\pi]$ into $I_1
= [0,\varepsilon]$ and $I_2 = [\varepsilon,\pi]$ (we take $\varepsilon
= 0.1$ in all three cases $T_0,T_1,T_2$).
Calling $T_0(x)$ to the function in~\eqref{normT0} whose supremum on
$I$ gives $\|T_0\|$, it is clear that when $x \in I_1$
then $T_0(x)$ is bounded using the asymptotic expansion described in Lemma~\ref{normT0}.

We here explain the calculation when $x \in I_2$.
The first step consists on splitting the integral

\begin{align*}
T_0(x) = \frac{1}{2 x u_0}\int_{0}^{\pi}\big|
 K(x-y)+K(x+y)-2K(y)\big|y\,dy \\
 = \frac{1}{2 x u_0}\int_{0}^{x}\big|
   K(x-y)+K(x+y)-2K(y)\big|y\,dy \\
   +\frac{1}{2 x u_0}\int_{x}^{\pi}\big|
 K(x-y)+K(x+y)-2K(y)\big|y\,dy \\
 = T_{0}^{1}(x) + T_{0}^{2}(x).
\end{align*}

The expression of $T_{0}^{2}(x)$ can be explicitly calculated (see
equation~\eqref{intT00}) so we will focus in the calculation of $T_{0}^{1}(x)$. Changing variables, we will write

\begin{align*}
T_{0}^{1}(x) & = \frac{x}{2 u_0}\int_{0}^{1}\big|
 K(x(1-w))+K(x(1+w))-2K(x w)\big|w\,dw \\
\end{align*}

We should note, however, that the integrand is singular (although integrable) at $w= 0$ and $w= 1$. The next step is to remove those singularities and treat them separately. We thus split $T_{0}^{1}(x)$ as
\begin{align*}
T_{0}^{1}(x)  &= \frac{x}{2 u_0}\int_{0}^{\de_0}\big|
               K(x(1-w))+K(x(1+w))-2K(x w)\big|w\,dw\\
             &  + \frac{x}{2 u_0}\int_{\de_0}^{1-\de_1}\big|
               K(x(1-w))+K(x(1+w))-2K(xw)\big|w\,dw
  \\
  &+ \frac{x}{2 u_0}\int_{1-\de_1}^{1}\big|
 K(x(1-w))+K(x(1+w))-2K(x w)\big|w\,dw \\
&= T_{0}^{1,1}(x) + T_{0}^{1,2}(x) + T_{0}^{1,3}(x)
\end{align*}
with $\de_0=\de_1 = 10^{-6}$ for $T_0,T_1$, and
$\de_0=0.0625, \de_1=10^{-4}$ for $T_2$.
The values of $T_{0}^{1,1}(x)$ and
$T_{0}^{1,3}(x)$ are calculated using asymptotic expansions at $w= 0$
and $w = 1$ respectively.
We remark that the integrand of $T_1^{1}(x)$ (the analog of
$T_0^1(x)$ for the operator $T_1$),
is not singular at $w= 0$ so we do not have to make that splitting of
the singularity. We are left with the calculation of $T_{0}^{1,2}(x)$,
which we pass to explain now for a fixed (interval) $x$.

In this case, the integration is done recursively. For each subdomain, we compute an enclosure of the integral. Since the integrand is not smooth because of the absolute value, we first compute a $C^0$ enclosure (i.e. evaluating the integrand at the full integration region). If the enclosure is sign-definite, the integrand is $C^2$ inside it, so we can improve on the width of the enclosure by performing a midpoint quadrature, given by:

\begin{align*}
\int_{a}^{b} f(y) dy \in  (b-a)f\left(\frac{a+b}{2}\right) + \frac{1}{24}(b-a)^3f''([a,b])
\end{align*}

We now decide to accept or reject the result, based on its width in an absolute and a relative (to the length of the integration region) way (it has to be smaller than \texttt{abs\_tol} and \texttt{rel\_tol} respectively). In the latter case, we split the region and recompute the integral on both subregions. The splitting is done by the midpoint. We keep track of the regions over which we need to integrate in a queue, implemented using a circular array. In order to avoid infinite loops -- which could potentially happen since there is uncertainty in the value of $x$ --, the size of the queue is limited at all times to \texttt{QSIZE} elements. In our code, \texttt{QSIZE} $= 1024$. If the program is not able to calculate an enclosure of the integral with the desired tolerances, we split in $x$ (by the midpoint) and recalculate each part until it meets them.

The integration region $I_2$ is further split into 3 regions. This is
because the source of the error comes from different places: for $x$
small, most of the error will come from the evaluation of $u$ and its
derivatives. For $x$ large, it will come from the integral. If $x$ is
close to $\pi$, we do not decide based on relative tolerances since
the result is very small (even 0). The different subregions were
$I_{2,1} = [0.1,1]$, $I_{2,2} = [1,3]$ and $I_{2,3} = [3,\pi]$. The
total runtime (for the 3 regions combined) was about 2 hours for
$T_0$, about 8 hours for $T_1$ and about 50 hours for $T_2$.

\end{proof}

\section*{Acknowledgements}

A.E.\ and B.V.\ are supported by the ERC Starting Grant~633152 and by the
ICMAT--Severo Ochoa grant SEV--2015--0554.
J.G.--S. was partially supported by the grant MTM2014--59488--P (Spain), by the ICMAT--Severo Ochoa grant SEV--2015--0554, by the Simons Collaboration Grant 524109 and by the NSF--DMS 1763356 Grant.

 \end{document}